\documentclass[reqno,centertags, 12pt]{amsart}
\usepackage{amsmath,amsthm,amscd,amssymb,latexsym,verbatim}

\usepackage{graphicx,epsf}

-\marginparwidth \oddsidemargin -4mm \evensidemargin \oddsidemargin

\newtheorem{theorem}{Theorem}[section]

\newtheorem{lemma}[theorem]{Lemma}

\theoremstyle{definition}
\newtheorem{definition}[theorem]{Definition}

\theoremstyle{remark}

\newcounter{smalllist}

\DeclareMathOperator*{\divg}{div}
 
\allowdisplaybreaks
\numberwithin{equation}{section}

\newcommand{\lb}{\label}

\newcommand{\beq}{\begin{equation}}
\newcommand{\eeq}{\end{equation}}

\newcommand{\bal}{\begin{align}}
\newcommand{\eal}{\end{align}}
\newcommand{\bals}{\begin{align*}}
\newcommand{\eals}{\end{align*}}

\newcommand{\bbN}{{\mathbb{N}}}
\newcommand{\bbR}{{\mathbb{R}}}

\newcommand{\bbP}{{\mathbb{P}}}

\newcommand{\bbZ}{{\mathbb{Z}}}

\newcommand{\bbT}{{\mathbb{T}}}

\newcommand{\calC}{{\mathcal C}}

\newcommand{\eps}{\varepsilon}
\newcommand{\del}{\delta}
\newcommand{\tht}{\theta}

\newcommand{\til}{\tilde}

\begin{document}
\title[Generalized Traveling Waves in Disordered Media]
{Generalized Traveling Waves in Disordered Media: Existence, Uniqueness, and Stability}

\author{Andrej Zlato\v s}

\address{\noindent Department of Mathematics \\ University of
Chicago \\ Chicago, IL 60637, USA \newline Email: \tt
andrej@math.uchicago.edu}


\begin{abstract}
We prove existence, uniqueness, and stability of transition fronts (generalized traveling waves) for reaction-diffusion equations in cylindrical domains with general inhomogeneous ignition reactions. We also show uniform convergence of solutions with exponentially decaying initial data to time translates of the front. In the case of stationary ergodic reactions the fronts are proved to propagate with a deterministic positive speed. Our results extend to reaction-advection-diffusion equations with periodic advection and diffusion.
\end{abstract}

\maketitle

\section{Introduction and Results} \lb{S1}

In this paper we study time-global solutions, called transition fronts or generalized  traveling waves, of reaction-diffusion equations on infinite cylinders. We consider the PDE
\beq \lb{1.0}
u_t  = \Delta u + f(x,u)
\eeq
which is used in modeling of processes such as autocatalytic chemical reactions, propagation of advantageous genes in a population, and combustion. The function $u(t,x)\in [0,1]$ is the (normalized) concentration of a reactant or allele, or temperature of  a combusting solid or gaseous medium. The non-negative reaction term $f$ accounts for increase of concentration/temperature due to a chemical reaction or burning and satisfies $f(x,0)=f(x,1)=0$. We will be particularly interested in {\it ignition reactions}, which vanish for $u$ smaller than some {\it ignition temperature} $\tht(x)>0$ and are used in the modeling of combustion, but we will also treat general non-negative reactions. The function $f$ will satisfy some uniform bounds but will otherwise be an arbitrary {\it non-periodic} function of $x$. 

We will also consider the more general equation
\beq \lb{1.1}
u_t + q(x)\cdot\nabla u = \divg(A(x)\nabla u) + f(x,u),
\eeq
with an incompressible mean-zero vector field $q$ representing advection and a uniformly elliptic diffusion operator $\divg (A\nabla)$ representing inhomogeneous diffusion. Unlike $f$, both $q$ and $A$ will be assumed periodic (with the same period).

Our main goal is the proof of existence and uniqueness of  transition fronts under very general conditions on $f$.  Moreover, we also want to prove uniform convergence of arbitrary solutions of \eqref{1.0}/\eqref{1.1} with exponentially decaying initial data to these fronts, thus describing the behavior of very general solutions of the PDEs. We will do this for \eqref{1.0}/\eqref{1.1} with ignition reactions, and also prove existence of fronts for some non-ignition reactions, on the cylindrical domain $D\equiv\bbR\times\bbT^{d-1}$ (i.e., $\bbR\times[0,1]^{d-1}$ with periodic boundary conditions). However, all our results can be extended to open connected domains $D\subseteq \bbR^d$ with a smooth boundary which are periodic in the first variable and bounded in the others, with either periodic or Neumann boundary conditions on $\partial D$ (the latter being $\nu\cdot\nabla u=0$ for \eqref{1.0} or $\nu \cdot A\nabla u=0$ and $q\cdot\nu=0$ for \eqref{1.1}, with $\nu$ the outward unit normal to $\partial D$). These include cylinders with periodically ondulating boundaries and a periodic array of holes. 

Such domains, unbounded in arbitrarily many variables, have been considered in \cite{BH} where transition fronts for {\it periodic} $f$ (as well as $q$ and $A$) have been studied. We restrict ourselves here to domains unbounded in only one variable because this is essentially the only case when transition fronts for ignition reactions can be unique even in homogeneous media (one moving right and one moving left). Moreover, the author has constructed examples of ignition reactions on $D=\bbR^2$ where no transition fronts exist \cite{Zla2}. In these $f$ is large on a sequence of concentric shells with exponentially growing radii and small elsewhere, and each non-trivial time-global solution is a {\it spatially extended pulse} (see \cite{BH2}) with $\|1-u(t,x)\|_{L^\infty_x}\to 0$ as $t\to\infty$. Nevertheless, questions about solutions of \eqref{1.0}/\eqref{1.1} on domains unbounded in several variables can also be treated by our methods. This will be done elsewhere \cite{Zla2}.

The following is the definition of a transition front from \cite{BH2}, adapted to our domain $D$.

\begin{definition} \lb{D.1.0}
A {\it transition front (moving to the right)} is a solution $w:\bbR\times D \to [0,1]$ of \eqref{1.0}/\eqref{1.1} that is global in time and satisfies for each $t\in\bbR$,
\beq \lb{1.2}
\lim_{x_1\to-\infty}w(t,x)=1 \qquad\text{and}\qquad
\lim_{x_1\to+\infty}w(t,x)=0
\eeq
uniformly in $(x_2,\dots,x_d)\in\bbT^{d-1}$. In addition, the front must have a {\it bounded width} (uniformly in time). That is, if $I_\eps(t)\subset\bbR$ is the smallest interval such that $w(t,x)\in [0,\eps]\cup [1-\eps,1]$ for $x\in D\setminus (I_\eps(t)\times\bbT^{d-1})$, then $L_{w,\eps}\equiv \sup_{t\in\bbR} |I_\eps(t)|<\infty$  for each $\eps>0$. The domain $I_{\eps_0}(t)\times\bbT^{d-1}$ for some small $\eps_0>0$ will be referred to as the {\it reaction zone}. 

We define a {\it transition front moving to the left} as above but with \eqref{1.2} replaced by
\[
\lim_{x_1\to-\infty}w(t,x)=0 \qquad\text{and}\qquad
\lim_{x_1\to+\infty}w(t,x)=1.
\]  
\end{definition}

{\it Remark.} In cylindrical domains fronts moving both right and left can exist. Since the transformation $x_1\mapsto -x_1$ interchanges the two directions, we will mostly consider fronts moving to the right.
\smallskip

The simplest case of transition fronts are  {\it traveling fronts} whose shape is time-independent. Their study goes back to the works of Kolmogorov, Petrovskii, Piskunov \cite{KPP} and Fisher \cite{Fisher} in 1937. They considered \eqref{1.0} in one spatial dimension $D=\bbR$ and with $x$-independent {\it KPP reaction} $f(u)$ (such that $0<f(u)\le f'(0)u$ for $u\in (0,1)$). In this case traveling fronts of the form $w(t,x)=W(x-ct)$ exist precisely when the {\it front speed} $c\ge c^*$, with the minimal speed being $c^*=2\sqrt{f'(0)}$. The {\it front profile} $W$ is time independent and satisfies the ODE $W_{xx}+cW_x+f(W)=0$ with $W(\infty)=0$ and $W(-\infty)=1$. The situation is the same for general {\it poisitive reactions} (such that $f(u)>0$ for $u\in(0,1)$) but the formula for $c^*>0$ is more complicated. In contrast, the front and its speed $c^*>0$ are unique for {\it ignition reactions} (such that $f(u)=0$ for $u\in[0,\tht]$ and $f(u)>0$ for $u\in (\tht,1)$, with $\tht>0$).

The ansatz $w(t,x)=W(x-cte_1)$ also works in more dimensions when $q,A,f$ are independent of $x_1$, and the answers are the same as above. In particular, the case of mean-zero shear flows $q$ and $x$-independent $A,f$ has been treated by Berestycki, Larrouturou, Lions \cite{Ber-Lar-Lions}, and Berestycki and Nirenberg \cite{Ber-Nir-2}. On the other hand, in the case of periodic $q,A,f$ (with the same period $p$), the front profile can only be expected to be time-periodic in a moving frame in the sense $w(t+p/c,x+pe_1)=w(t,x)$ for some speed $c>0$. Such {\it pulsating fronts} are of the form $w(t,x)=W(x_1-ct,x)$ with the profile $W$ decreasing in the first variable, periodic in the second variable, and satisfying
\[
\lim_{s\to-\infty} W(s,x)=1 \qquad\text{and}\qquad
\lim_{s\to+\infty} W(s,x)=0,
\]
uniformly in $x\in D$. Existence and uniqueness of pulsating fronts (with mean-zero $q$) was proved by Xin \cite{Xin3} for $x$-independent ignition reactions, and by Berestycki and Hamel \cite{BH} for $x$-periodic ignition reactions. The latter paper also treats $x$-periodic positive reactions and again proves existence of fronts with precisely the speeds $c\ge c^*$ for some $c^*>0$. In all these results it has been assumed that $\tht(x)\equiv \inf\{u>0\,|\, f(x,u)>0\}$ (called ignition temperature if it is positive) is $x$-independent (i.e., $\tht(x)=\tht$ for all $x\in D$) and $f(x,u)>0$ for $u\in (\tht,1)$.

The situation is different for disordered media, when no such ansatz exists and one has to work directly with the original PDE. Constant or periodic front profiles cannot be expected and fronts need not have a well defined speed. The definition of a transition front above has been given by Berestycki and Hamel \cite{BH2} in a more general setting and on arbitrary domains.  An alternative definition has been given by Shen \cite{Shen}, who studied fronts in random one-dimensional media and established some sufficient conditions on their existence. This formalizes an earlier definition by Matano which essentially requires the profile of the front to be a continuous function of the medium near the reaction zone. 

Due to the above difficulties, the existence of transition fronts in general disordered media has so far only been proved for \eqref{1.0} in one dimension, for some ignition reactions with $x$-independent ignition temperatures. Namely, Nolen and Ryzhik \cite{NolRyz} and independently Mellet, Roquejoffre, Sire \cite{MRS} have proved that such fronts exist on $D=\bbR$ when the reaction satisfies $a_0f_0(u)\le f(x,u)\le a_1f_0(u)$ and $f_0'(1)<0$. Here $0<a_0\le a_1<\infty$  and $f_0$  is of ignition type with $f_0(u)>0$ if and only if $u\in(\tht,1)$, $\tht>0$. Moreover, Mellet, Nolen, Roquejoffre, Ryzhik \cite{MNRR} proved that if in addition $f(x,u)=a(x)f_0(u)$ with $a(x)\in[a_0,a_1]$, then the (right-moving) front is unique. 

The usage of 1-dimensional techniques plays an important role in \cite{MNRR,MRS,NolRyz}, as does the requirement of an $x$-independent ignition temperature $\tht$. The latter has, in fact, been assumed in all previous proofs of existence of traveling/pulsating/transition fronts. We present here a new method that can handle much more general reactions, works in any dimension, and in the presence of (periodic) $q$ and $A$. In particular, we prove existence of a unique transition front when $f$ lies between two arbitrary $x$-independent ignition reactions, with possibly different ignition temperatures. We also prove existence of fronts in the more general case when the upper bound is a positive reaction, only requiring a bound on its derivative at zero (these fronts are not unique in general, as is the case for homogeneous media). We note that the requirement of a bound of this type is necessary to guarantee existence and cannot be improved except possibly by a constant (see Remark 1 after Theorem \ref{T.1.2}).

Let us now state our main results. We will start with the special case \eqref{1.0}. We will assume the following hypotheses on $f$.

\medskip

{\it (H1):  The reaction $f$ is uniformly Lipschitz with constant $K$ and  lies between two $x$-independent reactions, one of ignition type and the other positive or ignition. More specifically, there are Lipshitz functions $f_0,f_1$, decreasing on $[1-\eps,1]$ for some $\eps>0$, such that $f_0(u)\le f(x,u)\le f_1(u)$ for $(x,u)\in D\times[0,1]$. In addition, $f_0(0)=f_0(1)=f_1(0)=f_1(1)=0$, there is $\tht\in(0,1)$ such that $f_0(u)=0$ for $u\in[0,\tht]$ and $f_0(u)>0$ for $u\in(\tht,1)$, and there is $\tht'\in[0,1)$ such that $f_1(u)=0$ for $u\in[0,\tht']$ and $f_1(u)>0$ for $u\in(\tht',1)$. }

\medskip

Assume that $c_0>0$ is the speed of the unique (right-moving) traveling  front for \eqref{1.0} with $f$ replaced by the $x$-independent reaction $f_0$ (this front is of the form $w(t,x)=W(x_1-c_0t)$).  We then obtain existence of a transition front for $f$ provided $f_1'(0)<c_0^2/4$. Note that if $f_1$ is an ignition reaction, then $f_1'(0)=0$ and this condition  is automatically satisfied. In this case we also prove uniqueness of the (right-moving) front and that this unique front is a global attractor of general exponentially decaying initial data. 

Our reaction $f$ can have an $x$-dependent ignition temperature $\tht(x)\in[\tht',\tht]$ and, moreover, we do not require $f(x,u)>0$ for all $u\in(\tht(x),1)$. Because of this,  we will need to impose a very natural hypothesis (which is automatically satisfied when $\tht'=\tht$) that $f(x,\cdot)$ does not vanish after it has become large enough (except at $u=1$). 

\begin{definition} \lb{D.1.0b}
If $\zeta>0$ and $g\in C([0,1])$ with 
$g(u)>0$  for $u\in(0,1)$ are such that for
\beq \lb{1.3a}
\alpha_f(x) \equiv \inf \big( \{ u\in (0,1) \,|\, f(x,u)\ge \zeta u \} \cup \{1\} \big) 
\eeq
we have $f(x,u)\ge g(u)$ when $u\in[\alpha_f(x),1]$, then we say that $f$  {\it $\zeta$-majorizes} $g$ (on $[0,1]$). 
\end{definition}
 
{\it Remark.} Note that if $\zeta>f_1'(0)$, then  $\alpha_f(x) = \min ( \{ u\in (0,1) \,|\, f(x,u)= \zeta u \} \cup \{1\} ) >\tht'$.
\smallskip

Here is our main result for \eqref{1.0}.

\begin{theorem} \lb{T.1.2}
Let $f$ satisfy the hypotheses (H1). Assume that $f$ $\zeta$-majorizes $g$ for some $\zeta<c_0^2/4$ and some $g$ as in Definition \ref{D.1.0b}. 

(i) If $f_1'(0)<c_0^2/4$, then there exists a  transition front $w$ for \eqref{1.0}  moving to the right with $w_t>0$ (and another moving to the left).

(ii) If $f_1$ is an ignition reaction (i.e., $\tht'>0$) and $f$ is non-increasing in $u$ on $[\tht'',1]$ for some $\tht''<1$, then there is a unique (up to time shifts) transition front $w_+$ for \eqref{1.0} moving to the right (and another $w_-$ moving to the left). 

(iii) In the setting of (ii) we have convergence of solutions with exponentially decaying initial data to time shifts of $w_\pm$ in the sense of Definition \ref{D.1.0a} below. Moreover, this convergence is uniform in $f,a,u$ --- the $s_\eps$ in the definition only depends on $f_0,f_1,\zeta,g,K,\tht'',Y,\mu,\nu,\eps$ and the $L_\nu$  on $f_0,\nu$.  
\end{theorem}

\begin{definition} \lb{D.1.0a}
Let $w_\pm$ be some right- and left- moving transition fronts for \eqref{1.0} (or \eqref{1.1}) on $D$. We say that {\it solutions with exponentially decaying initial data converge to time shifts of $w_\pm$} if the following hold for any $Y,\mu,\nu>0$ and $a\in\bbR$.
 
(a) If $u$ solves \eqref{1.0} (or \eqref{1.1}) with initial datum 
\[
(\tht+\nu)\chi_{(-\infty,a]}(x_1)\le u_0(x) \le e^{-\mu(x_1-a-Y)},
\]
then there is $\tau_u$ such that for every $\eps>0$ there is $s_\eps>0$ such that for each $t\ge s_\eps$,
\beq \lb{1.4}
\|u(t,x) - w_+(t+\tau_u,x)\|_{L^\infty_x} < \eps
\eeq
(and similarly for solutions exponentially decaying as $x_1\to -\infty$ and for $w_-$).

(b) There is $L_\nu<\infty$ such that if $L\ge L_\nu$ and $u$ solves \eqref{1.0} (or \eqref{1.1}) with initial datum 
\[
(\tht+\nu)\chi_{[a-L,a+L]}(x_1)\le u_0(x) \le \min\{e^{-\mu(x_1-a-L-Y)}, e^{\mu(x_1-a+L+Y)}\},
\]
then there are $\tau_{u,\pm}$ such that for every $\eps>0$ there is $s_\eps>0$ such that for each $t\ge s_\eps$,
\beq \lb{1.5}
\|u(t,x) - w_+(t+\tau_{u,+},x) - w_-(t+\tau_{u,-},x) + 1\|_{L^\infty_x} < \eps.
\eeq
\end{definition}

As mentioned before, transition fronts need not exist in general (and are typically not unique) when the domain $D$ is unbounded in more than one variable, even for ignition reactions.
We next make several remarks which illuminate the necessity of the assumptions in Theorem \ref{T.1.2} on cylindrical domains $D$, and thus show that our result is {\it qualitatively sharp}.
\smallskip

{\it Remarks.}  1. The main condition here is $f_1'(0)<c_0^2/4$. Some condition of this type is necessary for existence of fronts, as can be seen from \cite{RoqZla}. There Roquejoffre and the author provide examples with $D=\bbR$ and $f_1'(0)$ arbitrarily close to $c_0^2$ where no transition fronts exist! In these examples $f$ is of KPP type on some bounded interval in $D$ and ignition elsewhere on $D$, and each non-trivial time-global solution is a spatially extended pulse with $\|u(t,x)\|_{L^\infty_x}\to 0$ as $t\to -\infty$.
\smallskip

2. The  condition $f_1'(0)<c_0^2/4$ is equivalent to $2\sqrt{f_1'(0)}< c_0$, meaning that minimal-speed fronts for KPP reactions $f(u)$ with $f'(0)=f_1'(0)$ are slower than the front for $f_0(u)$.  It then follows that the graph of $f_1'(0)u$ must intersect that of $f_0(u)$ at some $u>0$ because the minimal front speed is monotone with respect to the reaction. Thus we cannot treat the case when $f_1$ is a KPP reaction, which is not surprising in the light of the previous remark. 
\smallskip

3. Some condition of non-vanishing of $f$ after it has become large is also needed, otherwise no transition front connecting 0 and 1 might exist. An example of such a situation can be obtained by taking $f(x,u)=f(u)$ with $f$ non-zero only on $(\tfrac 13,\frac 23)\cup(\tfrac 23, 1)$ and much larger on the first interval than on the second --- it can be then showed that only fronts connecting 0 and $\tfrac 23$ exist. Nevertheless, our $f(x,\cdot)$ can be essentially an arbitrary function on the interval $[0,\alpha_f(x)]$, only required to be Lipschitz and lie between $f_0$ and $f_1$. In particular, it can vanish on an arbitrary closed subset of $[0,\tht]$, unlike in previous works where the reaction is required to vanish precisely on $[0,\til\tht]\cup\{1\}$ for a fixed $x$-independent $\til\tht\in[0,1)$. Thus our result is new even in the case of  periodic $f$.
\smallskip

4. As mentioned earlier, in the general positive reaction case (i) transition fronts are not unique even in homogeneous media.
\smallskip

5. We also note that some decay assumption on $u_0$ in (iii) needs to be made. Indeed, it is not hard to show that if $f(x,u)=f_0(u)$ and $u_0$ decays slowly enough, then $u$ will ``overtake'' all time shifts of $w_+$.
\smallskip

6. $L_\nu<\infty$ in Definition \ref{D.1.0a}(b) is such that solutions of \eqref{1.0} with $f_0$ in place of $f$ and $u_0(x)\ge (\tht+\nu) \chi_{[-L_\nu,L_\nu]}$ are guaranteed to spread, that is, $u(t,x)\to 1$ locally uniformly in $x\in D$ as $t\to\infty$. This can be taken to be the $L_\nu$ from Lemma \ref{L.5.1}.
\smallskip


Theorem \ref{T.1.2} extends to the more general case of \eqref{1.1} with periodic $q$ and $A$ which satisfy the following hypotheses. 

\medskip 

{\it (H2):  The flow $q\in Lip(D)$ is incompressible $\nabla\cdot q\equiv 0$, $p$-periodic in $x_1$, and with mean-zero first coordinate $\int_{\calC} q_1(x)\,dx=0$ (where $\calC=[0,p]\times\bbT^{d-1}$).  The  matrix $A\in C^{1,1}(D)$ is symmetric, $p$-periodic, and with $\underline AI\le A(x)\le \bar AI$ for some $0<\underline A\le\bar A<\infty$ and all $x\in D$. }

\medskip

Again let $c_0>0$ be the speed of the unique (right-moving) pulsating front for \eqref{1.1} with $f$ replaced by $f_0$  \cite{BH}.  We also let 
$\zeta_0>0$ be such that the minimal pulsating front speed for \eqref{1.1} with $f$ replaced by $\zeta_0 u(1-u)$ is $c_0$ \cite{BHN1}. Equivalently, $\zeta_0$ is the unique positive number such that the right hand side of \eqref{2.8c} below with $\zeta$ replaced by $\zeta_0$ equals $c_0$. For the left-moving front we have a possibly different speed $c_0^->0$, and we define $\zeta_0^->0$ accordingly, with $e_1,q_1$ replaced by $-e_1,-q_1$ in \eqref{2.8a}. The condition $f_1'(0)<c_0^2/4$ is now replaced by  $f_1'(0)<\zeta_0$ and our main result for \eqref{1.1} is  as follows.

\begin{theorem} \lb{T.1.1}
Let $q,A,f$ satisfy the hypotheses (H1), (H2). 
Assume that $f$ $\zeta$-majorizes $g$ for some $\zeta<\zeta_0$ and some $g$  as in Definition \ref{D.1.0b}.

(i)  If  $f_1'(0)<\zeta_0$, then there exists a transition front $w$ for \eqref{1.1} moving to the right with $w_t>0$ (and another moving to the left when $\zeta_0$ is replaced by $\zeta_0^-$ in the hypotheses). 

(ii) If $f_1$ is an ignition reaction (i.e.,  $\tht'>0$) and $f$ is non-increasing in $u$ on $[\tht'',1]$ for some $\tht''<1$, then there is a unique (up to time shifts) transition front $w_+$ for \eqref{1.1} moving to the right (and another $w_-$ moving to the left). 

(iii) In the setting of (ii) we have convergence of solutions with exponentially decaying initial data to time shifts of $w_\pm$ in the sense of Definition \ref{D.1.0a}. Moreover, this convergence is uniform in $f,a,u$ --- the $s_\eps$ in the definition depends only on $q,A,f_0,f_1,\zeta,g,K,\tht'',Y,\mu,\nu,\eps$ and the $L_\nu$ on $q,A,f_0,\nu$.  
\end{theorem}

{\it Remarks.} 1. Although the front speed is not well defined in general disordered media, it is easy to see from our proof that the reaction zone of $w_\pm$ moves with speed $\ge c_0$ resp.~$\ge c_0^-$. We also prove an $f$-independent bound on the width of $w_\pm$ (see the remark after the proof of Lemma \ref{L.3.1} and the last paragraph of Section \ref{S3}). In addition, $w_+$ decays at an exponential rate $\ge\lambda_\zeta$ as $x_1\to\infty$ (and similarly $w_-$ as $x_1\to -\infty$) determined from \eqref{2.8c} (see \eqref{3.1}). In fact, $\lambda_\zeta$ can be replaced here by any $\lambda$ such that the fraction in \eqref{2.8c} is smaller than $c_0$. 
\smallskip


2. The proof of Theorem \ref{T.1.1} can, in fact, be made independent of previous results on transition fronts. In particular, we can prove that a unique pulsating front with speed $c_0$ exists for \eqref{1.1} with the $x$-independent reaction $f_0$ (which we need in order to state Theorem~\ref{T.1.1}). This is done in the proof of Lemma \ref{L.5.1} below. The only result we will need is convexity of the function $\kappa(\lambda)$ in  \eqref{2.8a} and $\kappa'(0)=0$, proved in \cite[Proposition 5.7(iii)]{BH}.
\smallskip 



We next note that the ignition fronts in (ii) satisfy the earlier mentioned definition of Matano where the shape of the front is a continuous function of the medium in the neighborhood of the reaction zone. This follows from the uniform convergence of general solutions to the front in (iii). We provide an application of this principle to periodic and random media.

\begin{theorem} \lb{T.1.3}
Assume the hypotheses of Theorem \ref{T.1.1}(ii) and that $f$ is also $p$-periodic in $x_1$. Then there is $c_\pm>0$ such that the transition front $w_\pm$ from that theorem satisfies $w_\pm(t+p/c_\pm,x\pm pe_1)=w_\pm(t,x)$. This $c_\pm$ is` the speed of $w_\pm$ (i.e., $w_\pm$ is time-periodic in the frame moving right/left at the speed $c_\pm$).
\end{theorem}

{\it Remark.} This result is new when the ignition temperature $\tht(x)$ is not constant, as well as when $f(x,u)$ is allowed to vanish for $u>\tht(x)$.

\begin{proof}
Consider only $w_+$. Let $u$ and $u'$ be solutions of \eqref{1.1} with initial data $u_0(x)$ and $u_0(x+p)$, with any $u_0$ as in the first part of (iii). Then $u,u'$ are $p$-translates of each other (in $x_1$), and so the same must be true about $w_+(t+\tau_u,x)$ and $w_+(t+\tau_{u'},x)$. The result follows with $c_+\equiv p/(\tau_u-\tau_{u'})$.
\end{proof}

Although fronts in disordered media do not have well-defined speeds in general, our results can be used to show that there is a deterministic asymptotic speed of fronts for \eqref{1.0}/\eqref{1.1} when $f$ is random (stationary and ergodic with respect to translations in $x_1$). The {\it asymptotic front speed} of a transition front $w$ is defined as 
\beq\lb{1.7}
c\equiv \lim_{|t|\to\infty} \frac {|\til X_w(t)|}{|t|},
\eeq
provided the limit exists. Here $\til X_w(t)$ is the first coordinate of some point in the reaction zone of the front, for instance, one could take  $\til X_w(t)$ such that $w(t,x)=\tfrac 12$ for some $x=(\til X_w(t),x_2,\dots,x_d)$. Such $\til X_w(t)$ may not be unique but this is not important because the requirement of a bounded width of the front $w$ shows that the limit in \eqref{1.7} is independent of the choice.

\begin{theorem} \lb{T.1.4}
Consider a probability space $(\Omega,\bbP,\mathcal{F})$ and let the random reaction $f_\omega(x,u)$  satisfy the hypotheses of Theorem \ref{T.1.1}(ii) (or Theorem~\ref{T.1.2}(ii)) uniformly in $\omega\in\Omega$. In addition, assume that $f_\omega$ is  stationary and ergodic. That is, assume that there is a group $\{\pi_k\}_{k\in \bbZ}$ (or $\{\pi_y\}_{y\in \bbR}$) of measure preserving transformations acting ergodically on $\Omega$ such that $f_{\pi_k\omega}(x,u)=f_\omega(x-kpe_1,u)$ (or $f_{\pi_y\omega}(x,u)=f_\omega(x-ye_1,u)$).  Then there is $c_\pm>0$ such that the $\omega$-dependent transition front $w_{\pm,\omega}$  from that theorem has asymptotic speed $c_\pm$ for almost all $\omega\in\Omega$.
\end{theorem}

{\it Remarks.} 1. Theorem \ref{T.1.1}(iii) then shows that solutions with (large enough) exponentially decaying initial data  almost surely spread with asymptotic speeds $c_+$ to the right and $c_-$ to the left.
\smallskip

2. Theorem \ref{T.1.4} has been proved in \cite{NolRyz} for \eqref{1.0} in the one-dimensional setting $D=\bbR$ and with the random reaction function $f_\omega(x,u)=a(x,\omega)f_0(u)$, where $a$ is bounded below and above by positive constants and $f_0$ is of ignition type.

\begin{proof}
Consider only $w_+$ and the case of $p$-periodic $q,A$. Let $v(x)$ be the function from Lemma \ref{L.2.1} below with $\til\tht$ chosen as at the beginning of Section \ref{S3}. Let $u_m$ solve \eqref{1.1} with initial condition $u_m(0,x)=v(x-mpe_1)$ (so that $(u_m)_t>0$). For integers $n\ge m$ define
\[
\tau_{m,n}(\omega) \equiv \inf \big\{ t\ge 0 \,\big|\, u(t,x)\ge v(x-npe_1) \text{ for all $x\in D$} \big\}.
\]
Then the proof of Theorem \ref{T.1.1} (more precisely, \eqref{3.3c} and \eqref{3.4c} below) shows that $\tau_{m,n}(\omega) \in [C_0(n-m),C_1(n-m)]$ for some $0<C_0<C_1<\infty$ and all $\omega$.

The comparison principle shows that $\tau_{m,n}(\omega) \le \tau_{m,k}(\omega) + \tau_{k,n}(\omega)$ when $m\le k \le n$. We also have $\tau_{m+k,n+k}(\pi_k\omega) = \tau_{m,n}(\omega)$ for $k\in\bbZ$. Since the group $\{ \pi_k \}_{k\in\bbZ }$ acts ergodically on $\Omega$, the subadditive ergodic theorem shows that there is $\tau_+\in[C_0,C_1]$ such that
\[
\tau_+ = \lim_{n\to\infty} \frac{\tau_{0,n}(\omega)} n = \lim_{n\to \infty} \frac{\tau_{-n,0}(\omega)} n 
\]
for almost all $\omega$. Uniform convergence (in $\omega$) of the solution $u_0$ to the front $w_{+,\omega}$ and the proof of Theorem \ref{T.1.1} (more precisely, \eqref{3.1} below) then show  that $c_+=p/\tau_+$ is the asymptotic speed of $w_{+,\omega}$ almost surely.
\end{proof}


Let us finish this introduction with a brief description of the proof of Theorem \ref{T.1.1}. In Section \ref{S2} we construct the front as a limit of a (sub)sequence of special solutions $u_n$ of \eqref{1.1}, increasing in time and initially (at a sequence of times $\tau_n\to -\infty$) supported increasingly farther to the left. The sequence $\tau_n$ is chosen so that the reaction zone for $u_n$ arrives at the origin at $t=0$. The main issue is to show that the $u_n$ have a {\it  uniformly (in $n$) bounded width}  in the sense of Definition \ref{D.1.0}. In Section \ref{S3} we show that if $w$ is any transition front, then each $u_n$ converges to a time shift of $w$ in $L^\infty_x$, and the rate of this convergence is {\it uniform in $n$}. This will show that any two fronts must be time shifts of each other. Finally, we will recycle this argument in Section \ref{S4} to show  $L^\infty_x$-convergence of $u_n$ to a time shift of any solution $u$ as in Theorem \ref{T.1.1}(iii), again at a {\it uniform rate (in $u$)}. Since $u_n$ also converges uniformly to a time shift of $w$, the same will be true for $u$. 

This also proves the special case, Theorem \ref{T.1.2}. Finally, in the Appendix we show how to make our proof independent of previous results by using our arguments to obtain a slight improvement of Lemma \ref{L.2.6}(ii) below, which is from \cite{Xin}.

The author thanks Jean-Michel Roquejoffre for useful discussions. Partial support by the NSF grant DMS-0632442 and an Alfred P. Sloan Research Fellowship is also acknowledged.



\section{Existence of Fronts for General Reactions} \lb{S2}

In this section we will prove Theorem \ref{T.1.1}(i) by finding a front moving to the right. Let us assume that the period of $q,A$ in $x_1$ is $p=1$ and thus the unit cell of periodicity is $\calC=\bbT^d$ (the general case is treated identically). We will  assume without loss that $\zeta>f_1'(0)$ and that there is $\sigma\in(0,\zeta-f_1'(0))$ such that $f$ $\zeta'$-majorizes $g$ for each $\zeta'>\zeta-\sigma$. This can be done because if $\zeta'>\zeta$, then $f$ also $\zeta'$-majorizes $g$, so we only need to change $\zeta$ to $(\max\{\zeta,f_1'(0)\}+\zeta_0)/2$ and pick $\sigma\equiv (\zeta_0-\max\{\zeta,f_1'(0)\})/4$.  The first assumption implies that $\inf$ can be replaced by $\min$  and $\ge$ by $=$ in \eqref{1.3a}, and $\alpha_f(x)$ is uniformly bounded away from 0 and 1 (see \eqref{2.8}). The second guarantees that if a subsequence of reactions $f_n$ which satisfy (H1)  converges locally uniformly to $f$ (each such sequence has a locally uniformly convergent subsequence) and each $f_n$  $\zeta'$-majorizes $g$ for each $\zeta'>\zeta-\sigma$, then this $f$ not only satisfies (H1) but also $\zeta'$-majorizes $g$ for each $\zeta'>\zeta-\sigma$. This claim would not be true with $\ge$  in place of $>$. We note that all constants in this section will depend on $q,A,f_0,f_1,\zeta,g,K$ (also on $\zeta_0,\sigma,c_0$ which already only depend on $q,A,f_0,\zeta$)  {\it but not on} $f$.

As in the one-dimensional case \cite{MRS,NolRyz,Shen}, we will look for the front as a limit of solutions with initial data specified at increasingly negative times and supported further and further to the left. These solutions will be monotonically increasing in time. We will therefore need

\begin{lemma} \lb{L.2.1}
For each $\til\tht\in(\tht,1)$ there exists a function $v(x)$ supported in $(-\infty,0)\times\bbT^{d-1}\subseteq D$ with $v(x) = \til\tht$ for all small enough $x_1$ such that
\beq \lb{2.1}
- \divg(A(x)\nabla v) + q(x)\cdot\nabla v \le  f_0(v)
\eeq
in the sense of distributions.
\end{lemma}

\begin{proof}
Take a nondecreasing function $\rho\in C(\bbR)\cup C^2(\bbR^+)$ with $\rho(v)=0$ for $v\le 0$,  $\rho(v)=v$ for $v\in[0,\tfrac {\tht+\til\tht}2]$, $\rho''(v)\le 0$ for $v\in [\tfrac {\tht+\til\tht}2,1]$, and $\rho(v)=\til\tht$ for  $v\ge 1$. 

Let $\til v<0$ be a $C^2$ solution of $-\divg(A\nabla \til v)+q\cdot \nabla\til v = q_1-\divg(Ae_1)$ on $\bbT^d$, periodically continued to $D$ (here  $e_1=(1,0,\dots,0)$ and $q_1= q\cdot e_1$). Such $\til v$ exists because the integral of the right hand side is zero, and because the left hand side annihilates constants. Then we let $v_\eps(x)\equiv \eps(\til v(x)-x_1)$ and $v(x)\equiv \rho(v_\eps(x))$ for $\eps>0$, so that $v$ is supported in $(-\infty,0)\times\bbT^{d-1}$. We have $-\divg(A\nabla v_\eps)+q\cdot \nabla v_\eps =0$ and so for some distribution $T\ge 0$ supported on the set $D_\eps\equiv\{x\in D\,|\, v_\eps(x)=0\}$,
\[
-\divg(A\nabla  v)+q\cdot \nabla v = -\eps^2 \chi_{D\setminus D_\eps} \rho''(v_\eps) (\nabla \til v-e_1) \cdot A (\nabla \til v-e_1)  - T.
\]
If $\rho''(v_\eps(x)) < 0$, then $v(x)=\rho(v_\eps(x))\in [\tfrac {\tht+\til\tht}2,\til\tht]$. Since $f_0$ is uniformly positive on this interval and $A$ is a positive matrix, \eqref{2.1} follows provided $\eps>0$ is small enough.
\end{proof}

We now fix $\til\tht<1$ to be close to 1 so that $\til\tht>\tht_0$ from \eqref{2.8} below (in particular, $\til\tht>\tht$) and consider the corresponding $v$ along with the functions $v_n(x)\equiv v(x+ne_1)$. For all $n\in\bbN$ let $u_n$ solve \eqref{1.1} for $t>\tau_n$ with initial condition $u_n(\tau_n,x)=v_n(x)$, where $\tau_n\to -\infty$ will be chosen shortly. We then have

\begin{lemma} \lb{L.2.2}
The functions $u_n$ satisfy for all $t> \tau_n$ and $x\in D$
\beq \lb{2.2} 
(u_n)_t (t,x)> 0
\eeq
as well as 
\beq \lb{2.3} 
\lim_{t\to\infty} u_n(t,x)=1
\eeq
locally uniformly in $x\in D$.
\end{lemma}

\begin{proof}
The time derivative $\dot u_n\equiv (u_n)_t$ satisfies $\dot u_n(\tau_n,x)\ge 0$  due to \eqref{2.1}  and periodicity of $q,A$, and it is not identically 0. Since $\dot u_n$ satisfies $(\dot u_n)_t + q\cdot\nabla \dot u_n = \divg(A\nabla  \dot u_n) + \tfrac{\partial f}{\partial u}(x,u_n)\dot u_n$ and $\tfrac{\partial f}{\partial u}$ is bounded, the strong maximum principle shows \eqref{2.2}.

This means that for each $n$  the function $\til u(x)\equiv \lim_{t\to\infty} u_n(t,x)$ is well defined and satisfies $\til u(x)\in(0,1]$  and 
\[
-\divg(A\nabla \til u ) + q\cdot\nabla \til u =  f(x,\til u).
\]
Lemma \ref{L.2.3}(ii) below now shows that $\til u$ is a constant, which is then 1 due to  $\|\til u\|_\infty\ge \til\tht>\tht$. Parabolic regularity then shows the limit to be locally uniform in $D$.
\end{proof}

We now choose $\tau_n<0$ to be the unique time such that
\beq \lb{2.4} 
u_n(0,0)=\tht.
\eeq
Note that $\tau_n\to -\infty$ (see the remark after Lemma \ref{L.2.4}).


Despite its apparent simplicity, we were not able to locate the following Liouville-type result in the literature.

\begin{lemma} \lb{L.2.3}
(i) Let $q,A$ be as in (H2) but without the assumptions of periodicity and $q_1$ being mean-zero. If the function $u$ is bounded on $D$ and  satisfies
\beq \lb{2.5} 
-\divg(A(x)\nabla  u )+ q(x)\cdot\nabla u =  0,
\eeq
then $u$ is constant.

(ii) Let $q,A$ be as in (H2) but without the assumption of periodicity, and let $r$ be a bounded non-negative function on $D$. If $u$ is bounded and non-negative on $D$ and satisfies
\beq \lb{2.6} 
-\divg(A(x)\nabla  u) + q(x)\cdot\nabla u =r(x)u,
\eeq
then $u$ is constant.
\end{lemma}

{\it Remark.} Note that if $q$ is incompressible on $D$, then its mean $\bar q\equiv \int_{\bbT^{d-1}} q(x_1,x')\,dx'$ is independent of $x_1$. Thus $\int_{[0,p]\times\bbT^{d-1}} q_1(x)\, dx=0$ is satisfied in (ii) for either all $p>0$ or none.

\begin{proof}
Let $u$ satisfy \eqref{2.5} or \eqref{2.6}. Then $\til u(x_1)\equiv \min_{x'\in \bbT^{d-1}} u(x_1,x')$ cannot have a local minimum by the maximum principle unless it is constant. In either case the limits 
\[
\lim_{x_1\to-\infty} \til u(x_1)=l_1 \qquad\text{and}\qquad \lim_{x_1\to\infty} \til u(x_1)=l_2
\]
exist. Harnack inequality for the domains $(y,y+1)\times\bbT^{d-1} \subset (y-1,y+2)\times\bbT^{d-1}$ with $y\to\pm\infty$ (see, e.g., \cite[p.~199]{GT}) now shows that uniformly in $x'=(x_2,\dots,x_d)$,
\[
\lim_{x_1\to-\infty} u(x)=l_1 \qquad\text{and}\qquad  \lim_{x_1\to\infty} u(x)=l_2.
\]
Parabolic regularity shows that $\lim_{x_1\to\pm\infty} \nabla u(x)=0$ uniformly in $x'$.

If $\bar q_1= 0$, integrate \eqref{2.6} over $D$ to get
\[
0=\int_D\divg(A\nabla  u) = \int_D \divg (qu) - \int_D ru = (l_2-l_1)\bar q_1 - \int_D ru = - \int_D ru.
\]
Thus $ru\equiv 0$ 
and \eqref{2.6} becomes \eqref{2.5}. We multiply \eqref{2.5} by $u$ and integrate over $D$ to get
\beq \lb{2.7}
-\int_D \nabla u \cdot A \nabla u = \int_D u \divg(A\nabla  u) = \frac 12 \int_D \divg (qu^2) = \frac 12 (l_2^2-l_1^2)\bar q_1 = 0.
\eeq
Thus $u$ must be constant, proving (ii) and the case $\bar q_1=0$ in (i).

If $\bar q_1\neq 0$ in (i), integrate \eqref{2.5} over $D$ to get
\[
0=\int_D \divg(A\nabla  u) = \int_D \divg (qu) = (l_2-l_1) \bar q_1.
\]
Thus $l_2=l_1$ and \eqref{2.7} finishes the proof of this case.
\end{proof}

We will now recover the transition front $w$ as a limit of the $u_n$ along a subsequence as $n\to\infty$. Such a limit always exists by parabolic regularity and satisfies $w(0,0)=\tht$ due to \eqref{2.4}, but the main issue is to show that it is indeed a transition front for \eqref{1.1}. The following four lemmas will ensure this fact.

Let us take $\zeta$ from the statement of Theorem \ref{T.1.1}  and let $\tht_j$ ($j=0,1$) be the smallest positive number such that $f_j(\tht_j)=\zeta\tht_j$. Since $\zeta>f_1'(0)$, we have $0<\tht_1\le\tht_0<1$, $\tht<\tht_0$, and for each $x\in D$,
\beq \lb{2.8}
\alpha_f(x) \in [\tht_1,\tht_0].
\eeq
We now let 
\beq \lb{2.8c}
c_\zeta \equiv \min_{\lambda>0} \frac {\zeta+\kappa(\lambda)} {\lambda}
\eeq
with $\kappa(\lambda)$ and $\gamma(x;\lambda)>0$ the principal eigenvalue and eigenfunction for
\beq \lb{2.8a}
\divg(A\nabla \gamma) - q\cdot\nabla\gamma - \lambda (Ae_1+A^T e_1)\cdot\nabla\gamma + \lambda(\lambda e_1^TAe_1 - \divg(Ae_1) + q_1)\gamma = \kappa(\lambda)\gamma
\eeq
on $\bbT^d$ (for \eqref{1.0} we have $\gamma(x;\lambda)\equiv 1$ and  $\kappa(\lambda)=\lambda^2$). We note that the minimum  is achieved at some $\lambda_\zeta>0$ because $\kappa$ is a continuous function of $\lambda$ and the fraction in \eqref{2.8c} diverges to $\infty$ as $\lambda\to 0,\infty$. The latter follows from $\zeta>0$ and 
\[
\kappa(\lambda) \ge \underline A\lambda^2,
\]
which is obtained after dividing \eqref{2.8a} by $\gamma$ and integrating over $\bbT^d$:
\[
\kappa(\lambda) = \int_{\bbT^d} (\nabla \log \gamma - \lambda e_1)\cdot A (\nabla \log \gamma - \lambda e_1) \ge \underline A \int_{\bbT^d} |\nabla \log\gamma|^2  + \lambda^2 \ge \underline A\lambda^2.
\]
We note that $c_\zeta<c_0$ because $\zeta<\zeta_0$ and $\zeta_0$ was defined so that the right hand side of \eqref{2.8c} with $\zeta_0$ in place of $\zeta$ equals $c_0$. 

We continue $\gamma(x;\lambda_\zeta)$ periodically on $D$, and define 
\[
\Psi(s,x)\equiv \left[ \inf_D \gamma(x;\lambda_\zeta) \right]^{-1} e^{-\lambda_\zeta s} \gamma(x;\lambda_\zeta) >0
\] 
(for \eqref{1.0} this is $\Psi(s,x)=e^{-\lambda_\zeta s}$).
Notice that $\Psi(0,x)\ge 1$, and $\psi(t,x)\equiv\Psi(x_1-c_\zeta t,x)$ is an (exponentially growing as $x_1\to -\infty$) pulsating front  with speed $c_\zeta$ for \eqref{1.1} with $f$ replaced by $\zeta u$. In fact, \cite{BHN1} shows that $c_\zeta$ is also the minimal speed of a true pulsating front for \eqref{1.1} with any $x$-independent KPP reaction $\til f$ satisfying $\til f'(0)=\zeta$. This is why $c_\zeta<c_0$ will be a crucial component of our argument (and, in fact, any $\lambda$ such that the fraction in \eqref{2.8c} is smaller than $c_0$ would do in place of $\lambda_\zeta$).

We now let for each $n\in\bbN$ and $t\ge\tau_n$,
\[
X_n(t)\equiv \sup \{x_1\,|\, u_n(t,x)\ge\alpha_f(x) \text{ for some $x=(x_1,x')$}\}
\]
and
\[
Y_n(t)\equiv \inf \{y\,|\, u_n(t,x)\le\Psi(x_1-y,x) \text{ for all $x\in D$}\}.
\] 
Both these functions are non-decreasing because $u_n$ is increasing and  $\Psi$ decreasing in their respective first variables.  Continuity of $u_n(t,x)$, lower semi-continuity of $\alpha_f(x)$, and compactness of $\bbT^{d-1}$ imply that $X_n$ is continuous from the right.

\begin{lemma} \lb{L.2.4}
Let $\xi\equiv \sup_{u\in(0,1)} f_1(u)/u\ge\zeta$ and $c_\xi\equiv (\xi+\kappa(\lambda_\zeta))/\lambda_\zeta$. Then for any $n$ and $t\ge \tau\ge \tau_n$ we have
\beq\lb{2.11}
Y_n(t)-Y_n(\tau)\le c_\xi(t-\tau). 
\eeq
\end{lemma}

{\it Remark.} Taking $\tau=\tau_n$ and $t=0$, this and \eqref{2.4} give $\tau_n\to -\infty$.

\begin{proof}
This is immediate from the definition of $Y_n(\tau)$ and the fact that he function $\phi(t,x)\equiv \Psi(x_1-Y_n(\tau)-c_\xi(t-\tau),x)$ is a supersolution of \eqref{1.1} for $t>\tau$ (it solves \eqref{1.1} with $\xi u$ in place of $f$).
\end{proof}

This and $Y_n(\tau_n)<\infty$ show that $Y_n(t)<\infty$ (then also $X_n(t)<\infty$ by \eqref{2.8}) and that $Y_n$ is continuous because it is non-decreasing. The following is a crucial step in our proof of existence of fronts.

\begin{lemma} \lb{L.2.5}
There is $C_1<\infty$  such that for all $n\in\bbN$ and all $t\ge \tau_n$,
\beq \lb{2.12}
|Y_n(t) - X_n(t)| \le C_1.
\eeq
\end{lemma}

\begin{proof}
The bound $X_n(t) - Y_n(t) \le C_1$ for a large enough $C_1$ is obvious from the definitions of $X_n,Y_n$, and \eqref{2.8}, so let us show $Y_n(t) - X_n(t) \le C_1$. We let $C_1$ (to be chosen later) be larger than  $C_0\equiv \max\{ Y_n(\tau_n)-X_n(\tau_n),1\}$, the latter independent of $n$ and finite due to $v_n$ being supported in a half-strip, $\til\tht>\tht_0$, and \eqref{2.8}.  Let us assume that $Y_n(t_1) - X_n(t_1)>C_1$ for some $n$ and $t_1\ge \tau_n$, and let $t_0\equiv \max\{ t<t_1 \,|\, Y_n(t)-X_n(t)\le C_0\}\ge\tau_n$. The maximum exists and $Y_n(t_0)-X_n(t_0)= C_0$ because $Y_n$ is continuous and $X_n$ non-decreasing.

Let $\til Y_n(t)\equiv Y_n(t_0)+c_\zeta(t-t_0)$ and if $X_n(t)\ge \til Y_n(t)$ for some $t\in[t_0,t_1]$, let $t_2$ be the first such time (recall that $X_n$ is continuous from the right). Then $\psi(t,x)\equiv \Psi(x_1-\til Y_n(t),x)$ is a solution and $u_n(t,x)$  a  subsolution in $\til D\equiv \{(t,x) \in(t_0,t_2)\times D \,|\, x_1>\til Y_n(t)\}$ of \eqref{1.1} with $\zeta u$ in place of $f$. Moreover, $\psi(t_0,x)\ge u_n(t_0,x)$ for $x_1>\til Y_n(t_0)$ and $\psi(t,x)\ge u_n(t,x)$ for $t\in[t_0,t_2]$ and $x_1=\til Y_n(t)$ (because $\Psi(0,x)\ge 1$). Thus $\psi(t,x)\ge u_n(t,x)$ in $\til D$ by the comparison principle, meaning that $Y_n(t)\le \til Y_n(t)$ for $t\in[t_0,t_2]$. But then $Y_n(t_2)-X_n(t_2)\le 0<C_0$, a contradiction with the choice of $t_0$. Hence such time $t_2\in[t_0,t_1]$ does not exist and the above argument gives $Y_n(t)  \le  \til Y_n(t)$ for $t\in[t_0,t_1]$.

That is, $Y_n$ increases with average speed at most $c_\zeta$ on $[t_0,t_1]$. On the other hand, the next lemma shows that $X_n$ increases with average speed at least $c_0-\eps$ (for any $\eps>0$) after an initial time delay $t_\eps$ (in the sense of \eqref{2.14} below with $\tau=t_0$). Since $c_\zeta<c_0$, we can pick $\eps\equiv (c_0-c_\zeta)/2>0$ and obtain 
\[
Y_n(t)-X_n(t)\le Y_n(t_0)+c_\zeta(t-t_0) - X_n(t_0) - (c_0-\eps)(t-t_0 -t_\eps) = C_0+(c_0-\eps)t_\eps- (c_0-\eps-c_\zeta)(t-t_0)
\]
for $t\in [t_0,t_1]$. If we now let $C_1\equiv C_0+c_0 t_\eps$, then it follows that $Y_n(t_1)-X_n(t_1)<C_1$, a contradiction. Thus $Y_n(t) - X_n(t) \le C_1$ for a large enough $C_1$ and all $n$ and $t\ge\tau_n$.
\end{proof}

Recall that $c_0^->0$ is the speed of the unique left-moving front for \eqref{1.1} with reaction $f_0$.

\begin{lemma} \lb{L.2.6}
(i) For every $\eps>0$ there is $t_\eps<\infty$ such that if $u:[0,\infty)\times D\to [0,1]$ solves \eqref{1.1}  with $u_t\ge 0$ and    $u(0,\til x)\ge\alpha_f(\til x)$ for some $\til x\in D$, then for each $t\ge 0$ we have
\[
\inf \big\{ u(t+t_\eps,x) \,\big|\, x_1-\til x_1\in [-(c_0^- -\eps)t,(c_0-\eps)t] \big\} \ge 1-\eps.
\]

(ii) There is $L'>0$ such that for every $\eps>0$ there is $t'_\eps<\infty$ satisfying the following. If $u:[0,\infty)\times D\to [0,1]$ solves \eqref{1.1}  and  $\inf \big\{ u(0,x) \,\big|\, |x_1-\til x_1|\le L' \big\} \ge (1+\tht)/2$ for some $\til x\in D$, then for each $t\ge 0$ we have
\[
\inf \big\{ u(t+ t'_\eps,x) \,\big|\, x_1-\til x_1\in [-(c_0^- -\eps)t,(c_0-\eps)t] \big\} \ge 1-\eps.
\]
\end{lemma}

{\it Remarks.} 1. Part (i) also shows that for $t\ge \tau\ge\tau_n$,
\beq \lb{2.14}
X_n(t)-X_n(\tau)\ge (c_0-\eps)(t-\tau -t_\eps).
\eeq
\smallskip

2. Of course, the constants $t_\eps,t_\eps'$ are independent of $f$.

\begin{proof}
 (ii) This is an immediate consequence of Proposition 3.4 in \cite{Xin} which proves the same result for $f(x,u)=f_0(u)$. In the appendix we will provide an alternative proof of this result (in fact, with $c_0-\eps,c_0^--\eps$ replaced by $c_0,c_0^-$), thus making our proof independent of \cite{Xin}.

(i) Consider $\sigma$  from the beginning of this section and let
\[
\alpha'_f(x)\equiv  \inf\{ u\in (0,1] \,|\, f(x,u)\ge (\zeta-\sigma/2) u \}.
\]
Then $\alpha'_f(x)\le\alpha_f(\til x)-\sigma\tht_1/4K$ for all  $x\in B_{\sigma\tht_1/4K}(\til x)$ by \eqref{2.8} (and $\zeta-\sigma/2>f_1'(0)$ shows that $\alpha'_f(x)$ is the minimum of the above set as well as uniformly bounded away from 0). So $u(0,\til x)\ge\alpha_f(\til x)$, $u_t\ge 0$, and parabolic regularity show that there is $\del>0$ (independent of $f,\til x$) such that 
\beq \lb{2.15}
u(1,x)\ge \alpha'_f(x) \qquad\text{for all $x\in B_{\del}(\til x)$.}
\eeq

Assume that (i) is false. Taking $L'$ from (ii) and using (ii), this means that for each $n\in\bbN$, there is a solution $w_n$ of \eqref{1.1} with some reaction $f_n$ satisfying all the hypotheses (in particular, $f_n$ $\zeta'$-majorizes $g$ for each $\zeta'>\zeta-\sigma$), such that $w_n(0,\til x^n)\ge \alpha_{f_n}(\til x^n)$, $(w_n)_t\ge 0$,  and  $w_n(n,x^n)< (1+\tht)/2$ for some $x^n$ with $ |x^n_1-\til x^n_1|\le L'$. After possible translation in $x_1$ it is sufficient to consider $\til x^n\in\bbT^d$, so we can assume $\til x^n\to\til x$ (otherwise we choose a subsequence). Then \eqref{2.15} gives $w_n(1,\til x)\ge \alpha'_{f_n}(\til x)$ for large $n$. 

By parabolic regularity, the functions $w_n$ are uniformly bounded in $C^{1,\eta;2,\eta}([1,\infty)\times D)$ for some $\eta>0$. Thus there is a subsequence (which we again denote $w_n$) converging in $C^{1;2}_{\rm loc}([1,\infty)\times D)$ to a solution $\til w\ge 0$ of \eqref{1.1} on $(1,\infty)\times D$ with some reaction $f=\lim_{n\to\infty} f_n$, satisfying $\til w_t\ge 0$. But then $w(x)\equiv \lim_{t\to\infty} \til w(t,x)$ exists and satisfies
\[
q\cdot\nabla w= \divg(A\nabla  w) + f(x,w).
\]
Lemma \ref{L.2.3} and boundedness of $f(x,u)/u$ show that $w$ is a constant. This constant is then 1 because $w(\til x)\ge \limsup_n \alpha'_{f_n}(\til x)$ and thus $f(\til x,w(\til x))\ge g(w(\til x))$ due to $f_n(\til x, u)\ge g(u)$ for $u\ge \alpha'_{f_n}(\til x)$. But $w_n(n,x^n)< (1+\tht)/2$ and $(w_n)_t\ge 0$ show for all $t\ge 1$,
\[
\|1-\til w(t,\cdot)\|_{L^\infty([\til x_1-L',\til x_1+L'])}\ge \frac{1-\tht}2>0.
\]
Parabolic regularity again shows that this contradicts $w\equiv 1$, thus finishing the proof.
\end{proof}

For $\eps>0$ we define
\begin{align*}
Z_{n,\eps}^-(t) & \equiv \sup \left\{ y \,\big |\, u_n(t,x)\ge 1-\eps \text{ when $x_1\le y$} \right\}, \\
Z_{n,\eps}^+(t) & \equiv \inf \left\{ y \,\big |\, u_n(t,x)\le \eps \text{ when $x_1\ge y$} \right\}.
\end{align*}
Clearly both are finite. The following will ensure a bounded width of the constructed transition front.

\begin{lemma} \lb{L.2.7}
For any  $\eps\in(0,\min \{c_0,c_0^-\})$ let $\til t_\eps\equiv t_\eps+C_1(\min \{c_0,c_0^-\} -\eps)^{-1}$ with $t_\eps$ from Lemma \ref{L.2.6}(i). Then there is $L_\eps<\infty$ such that for all $n$ and $t\ge\tau_n+\til t_\eps$,
\[
Z_{n,\eps}^+(t) - Z_{n,\eps}^-(t) \le L_\eps. 
\]
\end{lemma}

\begin{proof}
Notice that Lemma \ref{L.2.5} and continuity of $Y_n$ show that if $X_n$ has jumps, they cannot be larger than $2C_1$. This and $(u_n)_t>0$ mean that for each $t\ge \tau_n+\til t_\eps$ and any closed subinterval $I$ of $(-\infty,X_n(t-\til t_\eps)]$ of length $2C_1$ there is $x_1\in I$ and $x'\in\bbT^{d-1}$ with $u_n(t-\til t_\eps,x_1,x')\ge\alpha_f(x_1,x')$. Then Lemma \ref{L.2.6}(i) shows that $u_n(t,x)\ge 1-\eps$ whenever $x_1\le  X_n(t-\til t_\eps)+C_1$. On the other hand, Lemmas \ref{L.2.4} and \ref{L.2.5} show that $u_n(t,x)\le \eps$ whenever 
\[
x_1\ge  X_n(t-\til t_\eps)+C_1+c_\xi \til t_\eps+l_\eps,
\]
where $l_\eps$ is such that $\Psi(l_\eps,x)\le \eps$ for all $x\in D$. Thus $L_\eps\equiv c_\xi \til t_\eps+l_\eps$ works.
\end{proof}

Having Lemma \ref{L.2.7}, the proof of Theorem \ref{T.1.1} is now standard. Parabolic regularity shows that the functions $u_n$ are uniformly bounded in $C^{1,\eta;2,\eta}([\tau_n+1,\infty)\times D)$, so we can find a subsequence converging in $C^{1;2}_{\rm loc}(\bbR\times D)$ to a function $w$ on $\bbR\times D$ which then is also a solution of \eqref{1.1}. Moreover, \eqref{2.4} gives $w(0,0)=\tht$, which together with Lemma \ref{L.2.7}, \eqref{2.11}, and \eqref{2.12} ensures \eqref{1.2} as well as a bounded width of $w$. Thus $w$ is a transition front in the sense of Definition \ref{D.1.0}. The claim $w_t>0$ is immediate from \eqref{2.2} and the strong maximum principle for $w_t$. The exponential decay in Remark 1 follows from Lemma \ref{L.2.5}.

\section{Uniqueness of Fronts for Ignition Reactions} \lb{S3}

We will now prove Theorem \ref{T.1.1}(ii), again assuming that the period of $q$ in $x_1$ is $p=1$. Since now $f_1'(0)=0$, we have automatically $\zeta>f_1'(0)$, and we again assume that $f$ $\zeta'$-majorizes $g$ for each $\zeta'>\zeta-\sigma$. All constants in this section will depend on $q,A,f_0,f_1,\zeta,g,K,\tht''$ {\it  but not on $f$}. Without loss we only need to consider fronts moving to the right, which we will denote by $w$.

We can assume $\tht''\ge \tht_0$ (otherwise we change $\tht''$ to $\tht_0$) 
and let
\[
\eps_0\equiv \frac 12\min\{ \tht',1-\tht'' \} \qquad \text{and} \qquad \til\tht\equiv 1-\frac{\eps_0}2,
\]
thus fixing $v$ from Lemma \ref{L.2.1}.

We let $w$ be an arbitrary front for \eqref{1.1} (in particular, we do not assume $w_t>0$) and $u$ the solution of \eqref{1.1} with the fixed initial condition $v$. Our strategy is as follows. First we will show that $w$ has to decrease exponentially as $x\to\infty$ in the same way as $u_n$ in the last section. Then we will use this to show that $u$ has to converge to some time shift of $w$ in $L^\infty$ as $t\to\infty$ (thus any two fronts for $f$ must approach each other up to a time shift as $t\to\infty$). Finally, we will show that the rate of this convergence is uniform for all $f$ and all fronts $w$ with uniformly bounded width. This gives a uniform convergence of solutions $u_n$ with initial data $v_n(x)\equiv v(x+n)$ to time shifts of any front $w$, and the uniqueness of the front follows.

Let $w$ be an arbitrary transition front for \eqref{1.1}. Consider $\Psi$ from the last section and let
\begin{align*} 
X_w(t) &\equiv \sup \{x_1\,|\, w(t,x)\ge\alpha_f(x) \text{ for some $x=(x_1,x')$}\},  \\
Y_w(t) &\equiv \inf \{y\,|\, w(t,x)\le\Psi(x_1-y,x) \text{ for all $x\in D$}\},  \\
Z^-_{w,\eps}(t) &\equiv  \sup \left\{ y \,\big |\, w(t,x)\ge 1-\eps \text{ when $x_1\le y$} \right\}, \\
Z^+_{w,\eps}(t) &\equiv \inf \left\{ y \,\big |\, w(t,x)\le \eps \text{ when $x_1\ge y$} \right\}, 
\end{align*}
as well as
\[
 L_{w,\eps} \equiv \sup_{t\in\bbR} \{Z^+_{w,\eps}(t)-Z^-_{w,\eps}(t)\}, \qquad Z_w(t)\equiv Z^-_{w,\eps_0}(t), \qquad\text{and}\qquad   L_w\equiv L_{w,\eps_0}.
\]
All of these, except possibly $Y_w(t)$, are finite because $w$ is a transition front, and we have $X_w(t)\in [Z^-_{w,\eps}(t),Z^+_{w,\eps}(t)]$ for $\eps\le\eps_0 (\le\min\{\tht_1,1-\tht_0\})$. The next lemma shows $Y_w(t)<\infty$.

\begin{lemma} \lb{L.3.1}
There is $\til C_2<\infty$ (depending on  $L_{w}$ if $w_t\not\ge  0$) such that for all $t$ we have 
\beq \lb{3.1}
|Y_w(t)-Z_w(t)|\le \til C_2.
\eeq 
\end{lemma}

\begin{proof}
Again $Z_w(t)-Y_w(t)\le \til C_2$ (with a uniform bound) is immediate so we are left with proving $Y_w(t)-Z_w(t)\le \til C_2$. We fix any $\eps\in (0,\eps_0)$ and for $t\in\bbR$ define
\begin{align*} 
\alpha_{f,\eps}(x) & \equiv \inf\{ u\in (\eps,1] \,|\, f(x,u)\ge \zeta (u-\eps) \} \uparrow \alpha_f(x) \quad\text{as $\eps\to 0$}, \\
X_{w,\eps}(t) &\equiv \sup \{x_1\,|\, w(t,x)\ge\alpha_{f,\eps}(x) \text{ for some $x=(x_1,x')$}\} \downarrow X_w(t) \quad\text{as $\eps\to 0$},  \\
Y_{w,\eps}(t) &\equiv \inf \{y\,|\, w(t,x)\le\Psi(x_1-y,x)+\eps \text{ for all $x\in D$}\} \uparrow Y_w(t) \quad\text{as $\eps\to 0$}.
\end{align*}
The convergences hold because $\tht'>0$, $w$ and $f$ are continuous, and $\bbT^{d-1}$ is compact. Note that $\alpha_{f,\eps}(x)\in[\tht',\tht'']$ because $f(x,\cdot)$ decreases on $[\tht'',1]$. Thus for any $t$ and $\eps\le\eps_0$,
\beq \lb{3.2}
0\le X_{w,\eps}(t)-Z_w(t)\le L_w
\eeq
by $\eps_0<\tht',1-\tht''$.
Hence it is sufficient to show
\beq \lb{3.3}
Y_{w,\eps}(t)-X_{w,\eps}(t)\le C_2'
\eeq
with $C_2'$ independent of $\eps$ (then use \eqref{3.2} and take $\eps\to 0$ to obtain $Y_w(t)-Z_w(t)\le \til C_2\equiv C_2'+L_w$). We will prove \eqref{3.3} using the argument from Lemma \ref{L.2.5}. We do not have $w_t>0$ here but Lemma \ref{L.2.6}(i) will not be needed. Instead, Lemma \ref{L.2.6}(ii) will suffice thanks to \eqref{3.2}.

Pick any $t_0\in\bbR$ and notice that $\Psi(0,x)\ge 1$ implies $Y_{w,\eps}(t_0)\le Z^+_{w,\eps}(t_0)$. We also have $X_{w,\eps}(t_0)\ge Z_w(t_0)\ge Z^-_{w,\eps}(t_0)$ by \eqref{3.2}, so 
\[
Y_{w,\eps}(t_0)-X_{w,\eps} (t_0)\le L_{w,\eps}.
\]
As long as $X_{w,\eps}(t)\le Y_{w,\eps}(t)$ for $t\ge t_0$, the argument in Lemma \ref{L.2.5} shows that $Y_{w,\eps}(t)$ increases with average speed at most $c_\zeta$ (i.e., $Y_{w,\eps}(t)\le Y_{w,\eps}(t_0)+c_\zeta(t-t_0)$)  because $\psi(t,x)\equiv \Psi(x_1-Y_{w,\eps}(t_0)-c_\zeta(t-t_0),x)+\eps$ solves \eqref{1.1} with $\zeta(u-\eps)$ in place of $f$. On the other hand, Lemma \ref{L.2.6}(ii) means that $Z_w(t)$ (and thus also $X_{w,\eps}(t)$ due to \eqref{3.2}) increases with average speed at least $(c_0+c_\zeta)/2 > c_\zeta$ after an initial time delay $t_{(c_0-c_\zeta)/2}'$ (independent of $t_0$). Thus the faster moving $X_{w,\eps}(t)$ will catch up with $Y_{w,\eps}(t)$ and we have $X_{w,\eps}(t_1)\ge Y_{w,\eps}(t_1)$ for some $t_1\in[t_0,t_0+t_\eps'']$. Here $t_\eps''$ is independent of $t_0$ because the speed difference $\ge (c_0-c_\zeta)/2$ and initial distance $\le L_{w,\eps}$ for all $t_0$. After the time $t_1$ the argument of Lemma \ref{L.2.5} shows again that $Y_{w,\eps}(t)-X_{w,\eps}(t)$ must stay uniformly bounded above (independently of $\eps$). Indeed, $Y_{w,\eps}(t)$ is again continuous (using $\xi\equiv \sup_{u\in(\eps_0,1)} f_1(u)/(u-\eps_0)\ge\zeta$ in Lemma \ref{L.2.4}) and increases with average speed at most $c_\zeta$ when $X_{w,\eps}(t)\le Y_{w,\eps}(t)$. On the other hand, starting from any time $\tau\in\bbR$, $X_{w,\eps}(t)$ increases with average speed at leats $(c_0+c_\zeta)/2$  after an initial time delay $t_{(c_0-c_\zeta)/2}'$ (independent of $\eps$) due to \eqref{3.2} and Lemma \ref{L.2.6}(ii). This proves the existence of $C_2'$ (depending on $L_w$ but independent of $\eps$) such that \eqref{3.3} holds for all $t\ge t_0+t_\eps''$. Since $t_0$ has been arbitrary, \eqref{3.3} holds for all $t$ and all $\eps\in (0,\eps_0)$, and taking $\eps\to 0$ gives \eqref{3.1}. 

In particular, $Y_w(t)$ is finite, and as in Lemma \ref{L.2.4} we obtain for $t\ge \tau$,
\beq\lb{3.3b}
Y_w(t)-Y_w(\tau)\le c_\xi(t-\tau). 
\eeq

Finally, note that if $w_t\ge 0$, then Lemma \ref{L.2.6}(i) applies to $w$. As a result, the proof of  \eqref{3.3} for small enough $\eps$ is identical to that of \eqref{2.12} (using that  $f$ $(\zeta-\sigma/2)$-majorizes $g$). Hence we do not need \eqref{3.2} to show that $X_{w,\eps}(t)$ increases with average speed larger than $(c_0+c_\zeta)/2$, and $C_2'$ becomes $L_w$-independent. Then $\eps\to 0$ gives $Y_w(t)-X_w(t)\le C_2'$ and thus
\beq \lb{3.3a}
|Y_w(t)-X_w(t)|\le C_2'
\eeq
This and $L_w$-independent upper bounds on $X_w(t-t_{\eps_0})-Z_w(t)$ (from Lemma \ref{L.2.6}(i)) and on $X_w(t)-X_w(t-t_{\eps_0})$ (from \eqref{3.3a} and  \eqref{3.3b}) show that $\til C_2$ in \eqref{3.1} is also $L_w$-independent.
\end{proof}

{\it Remark.} Notice that this result, together with the definition of $Y_w$ and \eqref{3.4b} below, shows that $L_{w,\eps}$ depends only on $L_w$ and $\eps$.
\smallskip

Let $u$ be the solution of \eqref{1.1}  with initial condition $v(x)$ from Lemma \ref{L.2.1} (with fixed $\til\tht\equiv 1-\eps_0$). Let us define
\begin{align*} 
X_u(t) &\equiv \sup \{x_1\,|\, u(t,x)\ge \alpha_f(x) \text{ for some $x=(x_1,x')$}\},  \\
Y_u(t) &\equiv \inf \{y\,|\, u(t,x)\le\Psi(x_1-y,x) \text{ for all $x\in D$}\}, \\
Z^-_{u,\eps}(t) &\equiv  \sup \left\{ y \,\big |\, u(t,x)\ge 1-\eps \text{ when $x_1\le y$} \right\}, \\
Z^+_{u,\eps}(t) &\equiv \inf \left\{ y \,\big |\, u(t,x)\le \eps \text{ when $x_1\ge y$} \right\},
\end{align*}
as well as
\[
L_{u,\eps} \equiv \sup_{t\ge t'_\eps} \{Z^+_{u,\eps}(t)-Z^-_{u,\eps}(t) \}, \qquad Z_u(t) \equiv  Z^-_{u,\eps_0}(t) \qquad\text{and}\qquad L_u\equiv L_{u,\eps_0},
\]
with $t'_\eps$ from Lemma \ref{L.2.6}(ii). All these are finite as in Section \ref{S2} and again
\beq\lb{3.3c}
Y_u(t)-Y_u(\tau)\le c_\xi(t-\tau). 
\eeq
Lemmas \ref{L.2.5}--\ref{L.2.7} and $u_t>0$ again show 
\beq \lb{3.4a}
|Y_u(t)-Z_u(t)|\le C_2
\eeq
for some $f$-independent $C_2$. We also have that for each $\eps>0$ there are $C_\eps,\til C_\eps<\infty$ (the latter $L_w$-dependent if $w_t\not\ge  0$) such that for any $t\ge \tau$,
\begin{align} 
Z_{w,\eps}^-(t) & \ge Z_w(\tau) + \frac{c_0+c_\zeta}2 (t-\tau) - \til C_\eps, \lb{3.4b} \\
Z_{u,\eps}^-(t) & \ge Z_u(\tau) + \frac{c_0+c_\zeta}2 (t-\tau) - C_\eps \quad (\tau\ge t'_\eps \text{ if $\eps\le 1-\til\tht = \tfrac{\eps_0}2$}). \lb{3.4c}
\end{align}
Here \eqref{3.4c} holds because Lemma \ref{L.2.6}(ii) shows that 
\[
Z_{u,\eps}^-(t)  \ge Z_u(\tau-t'_\eps) + \frac{c_0+c_\zeta}2 (t-\tau)
\]
and $Z_u(\tau)-Z_u(\tau-t'_\eps)$ is uniformly bounded in $\tau$ due to \eqref{3.3c} and  \eqref{3.4a}. The same argument works for $w$ but it uses \eqref{3.1} and so $\til C_\eps$ depends on $L_w$ via $\til C_2$ (unless $w_t\ge 0$).

Before we can show that $u$ converges to some time shift of $w$, we need to prove that once $u$ is close to a time shift of $w$, it will not depart far from it.

\begin{lemma} \lb{L.3.2}
For each $\eps>0$ there is a $\del>0$ (depending also on $L_w$ if $w_t\not\ge  0$) such that:

(i) If $w(t_1,x)\le u(t_0,x) + \del$ for some $t_0\ge1$, $t_1\in\bbR$ and all $x\in D$, then $w(t+t_1-t_0,x) \le u(t,x) + \eps$ for all $t\ge t_0$ and $x\in D$.

(ii) If $w(t_1,x) \ge u(t_0,x) - \del$ for some $t_0\ge 1$, $t_1\in\bbR$ and all $x\in D$, then $w(t+t_1-t_0,x) \ge u(t,x) - \eps$ for all $t\ge t_0$ and $x\in D$.
\end{lemma}

\begin{proof}
This is proved via a construction of suitable supersolution and subsolution. To do that, let $\kappa(\lambda_\zeta/2)\ge 0$ and $\gamma(x;\lambda_\zeta/2)>0$ be from \eqref{2.8a} with $\lambda=\lambda_\zeta/2$. We then have from the convexity of $\kappa(\lambda)$ \cite[Proposition 5.7(iii)]{BH} and $\kappa(0)=0$,
\[
 \kappa \left( \frac{\lambda_\zeta}2 \right) \le \frac { \kappa(\lambda_\zeta)} {2}  <  \frac{c_\zeta \lambda_\zeta}2.
\]
This means that if we continue $\gamma(x;\lambda_\zeta/2)$ periodically on $D$ and let
\[
\Phi(s,x)\equiv \frac {\sup_D \gamma(x;\lambda_\zeta)} { \inf_D \gamma(x;\lambda_\zeta) \inf_D \gamma(x;\lambda_\zeta/2)} e^{-\lambda_\zeta s/2} \gamma(x;\lambda_\zeta/2) >0,
\] 
then 
\beq \lb{3.5}
\Phi(s,x)\ge e^{\lambda_\zeta s/2}\Psi(s,x)
\eeq
for $(s,x)\in \bbR\times D$, and for each $y\in\bbR$ the function $\phi(t,x)\equiv\Phi(x_1-y-c_\zeta t ,x)$ satisfies
\beq \lb{3.6}
\phi_t + q\cdot\nabla\phi - \divg(A\nabla  \phi) = \left[ \frac{c_\zeta \lambda_\zeta}2 - \kappa \left( \frac{\lambda_\zeta}2 \right)  \right] \phi \ge 0.
\eeq

Next pick $\omega\le 1\le \Omega$ so that for each $f$ as in the statement of Theorem \ref{T.1.1} (with fixed $q,A,f_0,f_1,\zeta,g,K,\tht''$),
\begin{align}
0<\omega\le &\inf \{u_t(t,x) \,|\, t\ge 1\text{ and } x_1\in [Z_u(t), Z_{u,\eps_0}^+(t)] \},  \lb{3.7} \\
\infty> \Omega\ge &\sup \{u_t(t,x) \,|\, t\ge 1\text{ and } x\in D \}.  \lb{3.8} 
\end{align}
The existence of such $\Omega$ follows from parabolic regularity and boundedness of $u$. The existence of $\omega$ is guaranteed by $u_t>0$ and is proved as follows. Assume the contrary, that is, there are sequences $f_n$ and $(t_n, x^n)\in [1,\infty)\times [Z_u(t_n), Z_{u,\eps_0}^+(t_n)]\times\bbT^{d-1}$  such that $u_t(t_n,x^n)\to 0$. 
As at the end of Section \ref{S2}, the functions $u_n(t,x)\equiv u(t+t_n,x+\lfloor x^n_1\rfloor e_1)$  contain a subsequence which converges in $C^{1;2}_{\rm loc}$ to a solution $\til u$ of \eqref{1.1}  on $(-1,\infty)\times D$ (with the same $q,A$ and some Lipschitz reaction $\til f(x,u)\in [f_0(u),f_1(u)]$ which is a locally uniform limit of a subsequence of $f_n(x+\lfloor x^n_1\rfloor e_1,u)$). But then $\til u_t(0,\til x)=0$ for some $\til x\in\bbT^d$, and so $\til u_t\ge 0$ and the strong maximum principle for $\til u_t$ show $\til u_t\equiv 0$. Since $Z_{u,\eps_0}^+(t)-Z_u(t)$ is uniformly bounded (in $t$ and $f$) due to \eqref{3.4a}, we have $\limsup_{x_1\to\infty} \til u(t,x)\le\eps_0$ and $\liminf_{x_1\to -\infty} \til u(t,x)\ge 1-\eps_0$. This contradicts $\til u_t\equiv 0$ because
\[
Z_{\til u}(t) \equiv \sup \{y \,|\, \til u(t,x)\ge 1-\eps_0 \text{ when $x_1\le y$}\}
\]
again grows with a positive average speed after an initial time delay, by Lemma \ref{L.2.6}(ii). 

Finally, assume without loss that $\eps\le\eps_0$ and $t_1=t_0$ (otherwise we shift $w$ in $t$ by $t_1-t_0$), and increase $K$ so that
\beq \lb{3.8a}
K\ge \frac {\lambda_\zeta(c_0-c_\zeta)} 4.
\eeq

(i) Notice that \eqref{3.4a} and the hypothesis show that (after possibly increasing $C_2$ by a constant only depending on $\Psi$ and thus not on $f,u,w$)
\beq \lb{3.9}
Z_u(t_0)\ge Z_w(t_0) - C_2.
\eeq
Let
\begin{align}
C_\eps' & \equiv C_2+\til C_2 + C_\eps+\til C_\eps+1, \notag\\
\beta(t) & \equiv \frac \eps {\Omega} \left( 1-e^{-\lambda_\zeta(c_0-c_\zeta)(t-t_0)/4} \right), \lb{3.10}\\
\phi_+(t,x) & \equiv \frac { \eps \lambda_\zeta(c_0-c_\zeta)  e^{-\lambda_\zeta C_\eps'/2} \omega}{4\Omega K \sup_D \Phi(0,x)} \Phi(x_1-Y_w(t_0)-c_\zeta(t-t_0),x) \lb{3.11}
\end{align}
so that \eqref{3.6} holds for $\phi_+$, and define for $t\ge t_0$,
\[
z_+(t,x)\equiv \til u_+(t,x)+\phi_+(t,x) \equiv u(t+\beta(t),x) + \phi_+(t,x).
\]
Our aim is to show
\beq \lb{3.12a}
z_+(t,x)\ge w(t,x)
\eeq
for all $x\in D$ and $t\ge t_0$. This estimate might not appear very useful because $\phi_+$ is unbounded but it will suffice. The reason is that at $t=t_0$, the function $\phi_+$ is large only where both $u,w$ are close to 1 and therefore also to each other. This setup will persist for all $t\ge t_0$ because $\phi_+$ travels with speed $c_\zeta$ which is strictly smaller than the speeds of propagation of $u$ and $w$. Thus, in fact, $\phi_+$ decays near the reaction zones of $u,w$ as $t$ grows.

Let $b_\eps\,(\le \eps)$ be the fraction in \eqref{3.11} and consider $\del\le b_\eps \inf_D \Phi (2 |\ln b_\eps|/\lambda_\zeta ,x )$. Then 
\beq \lb{3.13}
z_+(t_0,x)\ge w(t_0,x)
\eeq
for $x_1\le Y_w(t_0) + 2|\ln b_\eps|/\lambda_\zeta$ by the hypothesis and for $x_1\ge Y_w(t_0) + 2|\ln b_\eps|/\lambda_\zeta$ by the definition of $Y_w(t_0)$ and by $b_\eps \Phi (2 |\ln b_\eps|/\lambda_\zeta +y ,x ) \ge \Psi (2 |\ln b_\eps|/\lambda_\zeta +y,x )$ when $y\ge 0$ (see \eqref{3.5}). 

Moreover, we will prove that $z_+$ is a supersolution of \eqref{1.1} for $t\ge t_0$, with $f(x,u)=0$ when $u\ge 1$. Since \eqref{3.6} gives
\[
(z_+)_t+q\cdot\nabla z_+ - \divg(A\nabla  z_+) \ge f(x,z_+) + \left[ f(x,\til u_+)-f(x,z_+) + \beta' (t) u_t(t+\beta(t),x) \right],
\]
this will be established if we show that the square bracket is non-negative. 

This is clearly true for $x_1\le Z_u(t+\beta(t))$ since then $\til u_+(t,x)\ge \tht''$ and so $f(x,\til u_+)\ge f(x,z_+)$. Next, \eqref{3.4c}, \eqref{3.9}, \eqref{3.1}, and $\beta(t)\ge 0$ for $t\ge t_0$ give
\beq \lb{3.14a}
Z_{u,\eps}^-(t+\beta(t))-Y_w(t_0)-c_\zeta(t-t_0) \ge \frac{c_0-c_\zeta}2 (t-t_0) - C_\eps',
\eeq 
so for $x_1\ge Z_{u,\eps}^-(t+\beta(t))$ we find using \eqref{3.8a},
\beq \lb{3.14b}
\phi_+(t,x)\le b_\eps e^{-\lambda_\zeta [ (c_0-c_\zeta)(t-t_0)/2 - C_\eps'] /2} \sup_D \Phi(0,x) = \frac {\beta'(t)\omega} K \le \eps \frac{\lambda_\zeta(c_0-c_\zeta)}{4K} \le \eps.
\eeq
This gives for $x_1\ge Z_{u,\eps}^-(t+\beta(t))$,
\[ 
|f(x,\til u_+)-f(x,z_+)| \le K \phi_+ \le  \beta'(t)\omega.
\] 
Thus the square bracket is again non-negative for $Z_u(t+\beta(t))\le x_1\le  Z_{u,\eps_0}^+(t+\beta(t))$ due to \eqref{3.7} and $Z_u(t+\beta(t))\ge Z_{u,\eps}^-(t+\beta(t))$. The same is true for $x_1\ge Z_{u,\eps_0}^+(t+\beta(t))$ because then \eqref{3.14b} implies $z_+(t,x)\le \eps_0+\eps\le 2\eps_0 \le \tht'$, yielding $f(x,z_+)=0$.

Hence $z_+$ is a supersolution of \eqref{1.1} with \eqref{3.13}, meaning that \eqref{3.12a} holds. Thus for $x_1\ge Z_{u,\eps}^-(t+\beta(t))$,
\[
u(t,x)-w(t,x)\ge z_+(t,x)-w(t,x) -\phi_+(t,x)- \beta(t) \Omega  \ge 0 - \eps -\eps = - 2 \eps
\]
using \eqref{3.14b} and \eqref{3.10}, and for $x_1\le Z_{u,\eps}^-(t+\beta(t))$,
\[
u(t,x)-w(t,x)\ge \til u_+(t,x)-\beta(t)\Omega -1 \ge \til u_+(t,x)-\eps-1 \ge -2\eps.
\]
This proves (i) with $2\eps$ in place of $\eps$. Note that $\delta$ also depends on $C_\eps'$ and thus on $L_w$ when $w_t\not\ge 0$.

(ii) Recall that we assume $\eps\le\eps_0$, $t_1=t_0$, and \eqref{3.8a}, and let us also assume $\eps\le c_0^{-1}$. This time  \eqref{3.1} and the hypothesis give (after increasing $\til C_2$ by an $f,u,w$-independent constant)
\beq \lb{3.14c}
Z_w(t_0)\ge Z_u(t_0)-\til C_2.
\eeq
Then the proof goes along the same lines as in (i) but using
\begin{align}
\phi_-(t,x) & \equiv \frac { \eps \lambda_\zeta(c_0-c_\zeta)  e^{-\lambda_\zeta C_\eps'/2} \omega}{4\Omega K \sup_D \Phi(0,x)} \Phi(x_1-Y_u(t_0)-c_\zeta(t-t_0),x), \lb{3.15a} \\
z_-(t,x) &\equiv \til u_-(t,x)-\phi_-(t,x) \equiv u(t-\beta(t),x) - \phi_-(t,x). \notag 
\end{align}
This time 
\[
(z_-)_t+q\cdot\nabla z_- - \divg(A\nabla  z_- ) \le f(x,z_-) - \left[ f(x,z_-) - f(x,\til u_-)  + \beta' (t) u_t(t-\beta(t),x) \right].
\]
with $f(x,u)=0$ for $u\le 0$, and 
\beq \lb{3.16}
z_-(t_0,x)\le w(t_0,x)
\eeq
if $\del$ is as in (i). We again need to show that the square bracket is non-negative.

For $x_1\ge  Z_{u,\eps_0}^+(t-\beta(t))$ we have $ \til u_-(t,x)\le\eps_0$, so $f(x,\til u_-)=0$ and the square bracket is non-negative. For $Z_u(t-\beta(t))\le x_1\le  Z_{u,\eps_0}^+(t-\beta(t))$ the same is true because
\beq \lb{3.16a}
Z_{u,\eps}^-(t-\beta(t))-Y_u(t_0)-c_\zeta(t-t_0) \ge \frac{c_0-c_\zeta}2 (t-t_0) - C_\eps'
\eeq 
(from \eqref{3.4c}, \eqref{3.4a}, and $\beta(t)\le \eps\le c_0^{-1}\le 2/(c_0+c_\zeta)$) again gives for $x_1\ge Z_{u,\eps}^-(t-\beta(t))$,
\begin{align}
\phi_-(t,x)\le b_\eps e^{-\lambda_\zeta [ (c_0-c_\zeta)(t-t_0)/2 - C_\eps'] /2} \sup_D \Phi(0,x) &= \frac {\beta'(t)\omega} K \le \eps \frac{\lambda_\zeta(c_0-c_\zeta)}{4K} \le \eps, \lb{3.16b} \\
|f(x,\til u_-)-f(x,z_-)| \le & K \phi_- \le  \beta'(t)\omega. \notag 
\end{align}
For $x_1\le Z_u(t-\beta(t))$ we have  $\til u_-(t,x)\ge 1-\eps_0$, so the bracket is non-negative as long as $\phi_-(t,x)\le\eps_0$ (because then $1-\tht''\le z_-(t,x)\le \til u_-(t,x)$). This means that $z_-$ is a subsolution of \eqref{1.1} where $\phi_-(t,x)\le\eps_0$. 

Since  \eqref{3.4b}, \eqref{3.14c}, and \eqref{3.4a} imply
\[
Z_{w,\eps}^-(t)-Y_u(t_0)-c_\zeta(t-t_0) \ge \frac{c_0-c_\zeta}2 (t-t_0) - C_\eps',
\]
\eqref{3.16b} also holds for $x_1\ge Z_{w,\eps}^-(t)$. Thus $z_-$ is a subsolution of \eqref{1.1} on the set where $\phi_-(t,x)\le \eps_0$ while on the complement of that set we have $x_1\le Z_{w,\eps}^-(t)$ and so
\[
w(t,x)\ge 1-\eps\ge 1-\eps_0\ge 1-\phi_-(t,x)\ge z_-(t,x).
\]
This together with \eqref{3.16} gives $z_-(t,x)\le w(t,x)$ for $t\ge t_0$ and $x\in D$. The rest of the proof is analogous to (i), with $Z_{w,\eps}^-(t)$ in place of $Z_{u,\eps}^-(t+\beta(t))$.
\end{proof}

\begin{lemma} \lb{L.3.3}
If
\beq \lb{3.17}
\tau_w \equiv \inf\{ \tau\,\big|\, \liminf_{t\to\infty} \inf_{x\in D} [w(t+\tau,x)-u(t,x)]\ge 0 \},
\eeq
then $-\infty<\tau_w<\infty$. Moreover, the infimum is also a minimum and so
\beq \lb{3.18}
\liminf_{t\to\infty} \inf_{x\in D} [w(t+\tau_w,x)-u(t,x)]\ge 0.
\eeq
\end{lemma}

\begin{proof}
The set in \eqref{3.17} is an interval $(a,\infty)$ for some $a\le\infty$ due to $u_t>0$.
Inequality \eqref{3.4b} shows $\lim_{t\to\infty} Z_w(t) =\infty$, and the properties of $v$ give the existence of $\tau<\infty$ such that $w(\tau,x)\ge u(0,x)$ for all $x\in D$. The comparison principle then shows $w(t+\tau,x)\ge u(t,x)$ for all $t\ge 0$, $x\in D$ and so $\tau_w<\infty$.

In the opposite direction we notice that \eqref{3.4b} and \eqref{3.1} give  $\lim_{t\to -\infty} Y_{w}(t) =-\infty$. Hence for each $\del>0$ the condition of Lemma \ref{L.3.2}(i) is satisfied with a large negative $t_1$ and $t_0 = t'_\del$, where $t'_\del$ is from Lemma \ref{L.2.6}(ii). Then Lemma \ref{L.3.2}(i) and \eqref{3.7} prove for all $t\ge t_0$ that $\inf_{x\in D} [w(t+t_1-t_0,x)-u(t+2\eps/\omega,x)] \le -\eps$, provided we choose $\eps>0$ small enough and then $\del>0$  according to Lemma \ref{L.3.2}(i). Thus $\tau_w>t_1-t_0-2\eps/\omega>-\infty$.

Hence $\tau_w$ is finite and then the infimum must be a minimum by  \eqref{3.8}.
\end{proof}


\begin{lemma} \lb{L.3.4}
We have
\beq \lb{3.19}
\lim_{t\to\infty} \|w(t+\tau_w,x)-u(t,x)\|_{L^\infty_x} =0.
\eeq

\end{lemma}

\begin{proof}
We can assume without loss of generality that $\tau_w=0$ (otherwise we shift $w$ in $t$). Then \eqref{3.18} reads
\beq \lb{3.20}
\liminf_{t\to\infty} \inf_{x\in D} [ w(t,x)-u(t,x) ] \ge 0
\eeq
and we are left with proving
\beq \lb{3.21}
\limsup_{t\to\infty} \sup_{x\in D} [ w(t,x)-u(t,x) ] \le 0.
\eeq
Assume this is not true. Then by Lemma \ref{L.3.2}(i), there is $\del_0>0$ such that for all $t\ge 1$,
\beq \lb{3.22}
\sup_{x\in D} [ w(t,x)-u(t,x) ] \ge\del_0.
\eeq
Moreover, the definition of $\tau_w=0$ and Lemma \ref{L.3.2}(ii) show that for each $\tau>0$ there is $\del_\tau>0$ such that  for all $t\ge 1$,
\beq \lb{3.23}
\inf_{x\in D} [ w(t-\tau,x) -u(t,x) ] \le -\del_\tau.
\eeq
Finally, we claim that $Z_w(t)-Z_u(t)$ stays bounded as $t\to\infty$. The lower bound follows from  \eqref{3.20} and $L_{u},L_{w}<\infty$. The upper bound follows from \eqref{3.23} for $\tau=1$, $L_{u,\del_1},L_{w,\del_1}<\infty$, and a uniform upper bound on $Z_w(t)-Z_w(t-1)$  (due to \eqref{3.1} and \eqref{3.3b}).

As before, there is a sequence $t_n\to\infty$ such that the functions $w(t+t_n, x+\lfloor Z_w(t_n)\rfloor e_1)$ and $u(t+t_n, x+\lfloor Z_w(t_n)\rfloor e_1)$ converge in $C^{1;2}_{\rm loc}(\bbR\times D)$ to two solutions $\til w, \til u$ of \eqref{1.1} with some reaction $\til f$ which has all the properties of $f$. Moreover, $\til w, \til u$ are both transition fronts because of the boundedness of  $Z_w(t)-Z_u(t)$ and the properties of $w,u$ (namely, \eqref{3.1}, \eqref{3.3b}, \eqref{3.4b}, \eqref{3.3c}, \eqref{3.4a}, and \eqref{3.4c}). We also have $\til u_t\ge 0$ as well as 
\begin{align}
\til w(t,x)  \ge \til u(t,x) & \quad\text{for all $(t,x)\in\bbR\times D$}, \lb{3.24} \\
\sup_{x\in D} [ \til w(t,x) - \til u(t,x) ]  \ge\del_0 & \quad\text{for all $t\in\bbR$},  \lb{3.25} \\
\inf_{x\in D} [ \til w(t-\tau,x) - \til u(t,x) ]  \le -\del_\tau & \quad\text{for all $t\in\bbR$, $\tau>0$.} \notag 
\end{align}
This is thanks to \eqref{3.20}, \eqref{3.22}, \eqref{3.23}, $t_n\to\infty$, and the uniform boundedness  in $t$  of $\max\{ Z_{w,\eps}^+(t), Z_{u,\eps}^+(t)\} - \min\{ Z_{w,\eps}^-(t), Z_{u,\eps}^-(t) \}$ (for any $\eps>0$). 

We define $Z_{\til w,\eps}^\pm(t)$, $Z_{\til u,\eps}^\pm(t)$, $L_{\til w,\eps}$, $L_{\til u,\eps}$ analogously to $Z_{w,\eps}^\pm(t)$, $L_{w,\eps}$.  Then $Z_{\til w,\eps}^\pm(t) \ge Z_{\til u,\eps}^\pm(t)$ for any $\eps>0$ by \eqref{3.24}, and  $Z_{\til w,\eps}^+(t) - Z_{\til u,\eps}^-(t)$ is uniformly bounded in $t$ because $Z_w(t)-Z_u(t)$ stays bounded as $t\to\infty$. 
We let $Z^+(t)\equiv \max \{ Z_{\til w,\eps_0}^+(t), Z_{\til u,\eps_0}^+(t+1)  \}$ and $Z^-(t)\equiv Z_{\til u,\eps_0}^-(t)$ so that $Z^+(t)-Z^-(t)$ is also uniformly bounded in $t$. Inequality \eqref{3.25} shows that for each $t$ there is $x^t\in[Z_{\til u,\del_0}^-(t), Z_{\til w,\del_0}^+(t)] \times \bbT^{d-1}$ such that 
\[
\til w(t,x^t) - \til u(t,x^t)\ge\del_0.
\]
Then \eqref{3.24} and Harnack inequality give the existence of $\del'>0$ such that 
\[
\til w(t,x)-\til u(t,x)\ge\del' \quad\text{ whenever } x_1\in  [Z^+(t), Z^-(t)],
\]
and so \eqref{3.8} yieds the existence of $\tau\in (0,1)$ such that
\[
\til w(t,x)\ge \til u(t+\tau,x) \quad\text{ whenever } x_1\in [Z^+(t), Z^-(t)].
\]

We finish the proof with an argument similar to \cite{MNRR}. We define $z(t,x)\equiv \til w(t,x) - \til u(t+\tau,x) \in C^{1,\eta;2,\eta}(\bbR\times D)$ and notice that $z$ then satisfies
\[
z_t+q\cdot\nabla z - \divg(A\nabla  z ) = r(t,x) z
\]
with $|r(t,x)|\le K$. We also have 
\begin{align}
z(t,x)\ge 0 & \quad \text{ when } x_1\in [Z^+(t), Z^-(t)], \lb{3.30} \\
\inf_{x\in D}  z(t,x)   \le -\del_\tau & \quad\text{ for each } t\in\bbR, \notag \\ 
r(t,x)\le 0 & \quad \text{ when } x_1 \not\in [Z^+(t), Z^-(t)], \lb{3.32}
\end{align}
the last inequality holding because $\til f$ is non-increasing outside $[\eps_0,1-\eps_0]$ and  $\til u_t\ge 0$. Moreover, $z(t,x)\to 0$ as ${\rm dist} (x_1, [Z^+(t), Z^-(t)])\to \infty$ uniformly in $t$ because $\til w, \til u$ are transition fronts and hence have bounded width. Let $(t_n,x^{t_n})$ be such that
\[
\lim_{n\to\infty} z(t_n,x^{t_n}) = \del''\equiv \inf_{(t,x)\in\bbR\times D}  z(t,x)  <0.
\]
Notice that we then have a uniform bound on ${\rm dist} (x^{t_n}_1, [Z^+(t_n), Z^-(t_n)])$. Again a subsequence of the sequence of functions $z(t+t_n,x+\lfloor x^{t_n}_1\rfloor e_1)$ converges in $C^{1;2}_{\rm loc}(\bbR\times D)$ to a function $\til z$ with
\[
\til z(0,\til x)=\del''=\inf_{(t,x)\in\bbR\times D}  \til z(t,x)  <0
\]
for some $\til x \in\bbT^{d-1}$, and satisfying (due to \eqref{3.30} and \eqref{3.32})
\[
\til z_t+q\cdot\nabla \til z- \divg(A\nabla  \til z) \ge 0 \quad \text{ where } \til z(t,x)\le 0.
\]
The strong maximum principle then forces $\til z(t,x)\equiv \del''<0$, a contradiction with the uniform boundedness of ${\rm dist} (x^{t_n}_1, [Z^+(t_n), Z^-(t_n)])$ and \eqref{3.30}. This proves \eqref{3.21} and we are done.
\end{proof}

Our final ingredient is the claim that the convergence in \eqref{3.19} is uniform in $f$ and $w$.

\begin{lemma} \lb{L.3.5}
For any $C>0$ and fixed $q,A,f_0,f_1,\zeta,g,K,\tht''$, the convergence in Lemma \ref{L.3.4} is uniform in all $f$ as above and all fronts $w$ with $L_w\le C$. 
\end{lemma}

{\it Remark.} We will see later that the hypothesis $L_w\le C$ is satisfied for some $C<\infty$ and all $f,w$. Thus the convergence is uniform in all $f,w$ as in Theorem \ref{T.1.1}(ii). 
\smallskip

\begin{proof}
Assume the contrary. Thus for some $C,\eps>0$ and each $n\in\bbN$, there are $w_n,u_n$ as in Lemma \ref{L.3.4} --- solving \eqref{1.1} with reactions $f_n$ (which satisfy the hypotheses of Theorem \ref{T.1.1}(ii) with uniform $q,A,f_0,f_1,\zeta,g,K,\tht''$) and with $\tau_{w_n}=0$ after a translation of $w_n$ in $t$ --- such that $L_{w_n}\le C$ and for some $t_n\to\infty$
\beq \lb{3.33}
 \|w_n(t_n,x)-u_n(t_n,x)\|_{L^\infty_x} >\eps.
\eeq
We will obtain a contradiction by finding a subsequence of $\{(f_n,w_n,u_n)\}_n$ which converges locally uniformly to $(f,w,u)$ such that $L_w\le C$ and \eqref{3.19} is violated. 

By parabolic regularity, for some $\eta>0$, the $w_n$ are uniformly bounded in $C^{1,\eta;2,\eta}(\bbR\times D)$ and the $u_n$ in $C^{1,\eta;2,\eta}([a,\infty)\times D)$ (for any $a>0$). We can thus choose a subsequence (which we again index by $n$) such that $f_n\to f$ in $C_{\rm loc}(D)$, 
$w_n\to w$ in $C^{1;2}_{\rm loc}(\bbR\times D)$ and $u_n\to u$ in $C^{1;2}_{\rm loc}((0,\infty) \times D)$. 
Therefore $w,u$ solve \eqref{1.1} on $\bbR\times D$ and $(0,\infty)\times D$, respectively.  Also, $u(0,x)= u_n(0,x)= v(x)$ holds because the $f_n$ are uniformly bounded and $v$ is continuous, so $\|u_n(t,x)- v(x)\|_{L^\infty_x}\to 0$ as $t\downarrow 0$, uniformly in $n$. We note that the limiting reaction $f$ again satisfies all the hypotheses, including $\zeta'$-majorization of $g$ for $\zeta'>\zeta-\sigma$.

Next we show that $w$ is a front. The $Z_{w_n}(0)$ must be uniformly bounded above because otherwise $w_n(-1,x)\ge v(x)$ for large $n$ and all $x\in D$, meaning that $\tau_{w_n}\le -1$, a contradiction. Similarly,  the $Y_{w_n}(0)$ are uniformly bounded below because of \eqref{3.3b} for $w_n$ and the fact that $Y_{w_n}(t'_\del)$ are uniformly bounded below for each $\del>0$ (by the argument in the second part of the proof of Lemma \ref{L.3.3} and $\tau_{w_n}=0$). Then Lemma \ref{L.3.1} and $L_{w_n}\le C$ show that $Z_{w_n}(0)$ and $Y_{w_n}(0)$ are uniformly bounded below and above, as are the average growth rates of $Z_{w_n}(t)$ and $Y_{w_n}(t)$ (due to \eqref{3.3b}, \eqref{3.4b}, and \eqref{3.1}). It follows from \eqref{3.4b} and $L_{w_n}\le C$ that the locally uniform limit $w$ is indeed a transition front with $L_w\le C$.

Thus Lemma \ref{L.3.4} applies to this $w$ and we have \eqref{3.19} for some $\tau_w$. So for each $\del>0$ there is $s_\del\ge t'_\del$ such that $\|w(s_\del+\tau_w,x)-u(s_\del,x)\|_{L^\infty_x} <\del$. Hence for each $M,\del>0$ and all large enough $n$ we have $\|w_n(s_\del+\tau_w,x)-u_n(s_\del,x)\|_{L^\infty_x(-M,M)} <2\del$. Since $Z_{w_n,\del}^-(s_\del+\tau_w)$, $Y_{w_n}(s_\del+\tau_w)$, $Y_{u_n}(s_\del)$ are uniformly bounded in $n$ by the argument above and $s_\del\ge t'_\del$, it follows that, in fact, $\|w_n(s_\del+\tau_w,x)-u_n(s_\del,x)\|_{L^\infty_x} < 2\del$ for all large enough $n$. Then $\del>0$ being arbitrary and Lemma \ref{L.3.2} show that for each $\eps'>0$ there are $N_{\eps'},r_{\eps'}$ such that for all $n>N_{\eps'}$ and $t> r_{\eps'}$,
\beq \lb{3.34}
 \|w_n(t+\tau_w,x)-u_n(t,x)\|_{L^\infty_x} < \eps'.
\eeq
Now \eqref{3.7} and $\eps_0\le \tfrac 14$ show for these $n,t$,
\[
\|w_n(t+\tau_w,x)-u_n(t+\tau_w,x)\|_{L^\infty_x} \ge \min \left\{ \tau_w \omega,\frac 12 \right\} - \eps'.
\]
If $\tau_w\neq 0$, then this contradicts $\tau_{w_n}=0$ and \eqref{3.19} when we take $\eps'$ small enough. Therefore $\tau_w=0$. But then after taking $\eps'= \eps$ and $n>N_\eps$ such that $t_n>r_\eps$, we obtain a contradiction between \eqref{3.33} and \eqref{3.34} with $t=t_n$.
\end{proof}

We can now proceed to prove Theorem \ref{T.1.1}(ii). Let $w$ be a transition front for \eqref{1.1} and translate it in $t$ so that $w(0,0)=\tht$. This is possible by \eqref{3.4b} although such translation may not be unique. Define $f_n(x,u)\equiv f(x-ne_1,u)$  and let $u_n$ solve \eqref{1.1} with reaction $f_n$ and $u_n(0,x)\equiv v(x)$. Pick $t_n$ so that $u_n(t_n,ne_1)=\tht$ and consider the front $w_n(t,x)\equiv w(t-t_n,x-ne_1)$ for \eqref{1.1} with $f_n$. 

We have $t_n\to\infty$ by \eqref{3.3b} as well as $L_{w_n}=L_w$ for each $n$. Then Lemma \ref{L.3.5} shows that for any $T\in\bbR$, uniformly in $t\ge T$,
\beq \lb{3.35}
\|w(t+\tau_{w_n},x)-u_n(t+t_n,x+ne_1)\|_{L^\infty_x}  = \|w_n(t+t_n+\tau_{w_n},x+ne_1)-u_n(t+t_n,x+ne_1)\|_{L^\infty_x}  \to 0
\eeq
as $n\to\infty$. This and $u_n(t_n,ne_1)=\tht$ show $w(\tau_{w_n},0)\to \tht$ as $n\to\infty$. Then $\tau_{w_n}$ must be bounded in $n$ by and $w(0,0)=\tht$, \eqref{3.4b}, and $L_w<\infty$. So there is a subsequence converging to some $\tau\in\bbR$. It follows from \eqref{3.8} and \eqref{3.35} that for each $t\in\bbR$,
\[
\|w(t+\tau,x)-u_n(t+t_n,x+ne_1)\|_{L^\infty_x}   \to 0
\]
along this subsequence. 

If now $w_1,w_2$ are two fronts for the same $f$, then we can choose the same subsequence for both, which gives the existence of $\tau_1,\tau_2$ such that $w_1(t+\tau_1,x)=w_2(t+\tau_2,x)$ for each $t,x$. Thus the two fronts are time shifts of each other, that is, each front is a time shift of the front $w$ constructed at the end of Section \ref{S2}.

Since this front satisfies $w_t>0$, the constants in this section do not depend on $L_w$. In particular, $\til C_2$ in \eqref{3.1} does not, which in turn gives a uniform in $f$ bound on $L_w$. This proves the remark after Lemma \ref{L.3.5}. 

\section{Stability of Fronts for Ignition Reactions} \lb{S4}

We will now prove Theorem \ref{T.1.1}(iii). We make the same assumptions as at the beginning of the last section and all constants will again depend on $q,A,f_0,f_1,\zeta,g,K,\tht''$ as well as on $Y, \mu,\nu$,  {\it but not on} $f$. Let us denote by $w_\pm$ the unique right- and left-moving fronts from the last section.


Without loss we will assume $\mu\le\lambda_\zeta/2$ and  \eqref{3.8a}. We can also assume $a=0$, after possibly shifting the domain. It will be notationally convenient to let $u$ be the solution of \eqref{1.1} with initial condition $v$ (with $\til\tht\equiv1-\eps_0/2$ as in the last section), and  $w$  the solution of \eqref{1.1} in question, with initial condition $w(0,x)\equiv w_0(x)$. 

Let us prove claim (a) in Theorem \ref{T.1.1}(iii) (i.e., Definition \ref{D.1.0a}(a) with $f,w_0$-uniform constants).  We have
\begin{align}
w_0(x) &\le e^{-\mu(x_1-Y)},  \lb{4.1}\\
w_0(x) &\ge \til\tht \chi_{(-\infty,0)}(x_1),  \lb{4.2}
\end{align}
where we have also assumed that $\nu= \til\tht-\tht$. This we can do without loss due to the following. Lemma \ref{L.2.6}(ii) in fact holds with $\tht+\nu$ in place of $(1+\tht)/2$ for any $\nu>0$ (see \cite{Xin} or Lemma \ref{L.5.1}) but with $\nu$-dependent $t_\eps'$ and $L'$. So if \eqref{4.2} holds with $\tht+\nu$ in place of $\til \tht$, then $w(\tau',x)\ge \til\tht \chi_{(-\infty,0)}(x_1)$ for $\tau'\equiv t_{1-\til\tht}'+L'/(c_0-\eps_0)$. We also have that if 
\beq \lb{4.2a}
Y_w(t) \equiv \inf \{y\,|\, w(t,x)\le\Phi(x_1-y,x) \text{ for all $x\in D$}\}
\eeq
with 
\[
\Phi(s,x)\equiv \frac {2 \sup_D \gamma(x;\lambda_\zeta)} { \inf_D \gamma(x;\lambda_\zeta) \inf_D \gamma(x;\mu)} e^{-\mu s} \gamma(x;\mu) >0,
\] 
corresponding to $\lambda=\mu$ in \eqref{2.8a}, then \eqref{3.3b} holds with $c_\xi=(\xi+\kappa(\mu))/\mu$ (and $\xi$ from Lemma \ref{L.2.4}). In particular, $Y_w(t)$ is again finite because \eqref{4.1} and 
\beq \lb{4.2b}
\Phi(s,x)\ge e^{-\mu s}
\eeq
imply $Y_w(0)\le Y$. Thus $w(\tau',x)\le e^{-\mu(x_1-Y-Y')}$ holds with the $f,w$-independent constant $Y'\equiv c_\xi\tau' + \mu^{-1}\sup_D \log \Phi(0,x)$. So \eqref{4.1} and \eqref{4.2} are satisfied for $w(\tau',x)$ in place of $w_0(x)$ and $Y+Y'$ in place of $Y$, with $\tau',Y'$ independent of $f,w$.

We define $X_w,Z^\pm_{w,\eps},L_{w,\eps},Z_w,L_w$ as before and $Y_w(t)$ by \eqref{4.2a}. The proof of claim (a) in Theorem \ref{T.1.1}(iii) will be essentially identical to the argument in Lemmas \ref{L.3.2}--\ref{L.3.5}, after we have established the basic properties \eqref{3.1}, \eqref{3.3b}, \eqref{3.4b} for $w$. 
We will then show uniform convergence of $u$ to a time shift of $w$. Since $u$ also uniformly converges to a time shift of $w_+$,  claim (a) in Theorem \ref{T.1.1}(iii) will thus be proved. The proof of claim (b) at the end of this section will be a slight variation on the same theme.


\begin{lemma} \lb{L.4.1}
The estimates \eqref{3.1}, \eqref{3.3b}, and \eqref{3.4b} hold with $f,u,w$-independent constants. 
\end{lemma}
 
\begin{proof}
We have already proved \eqref{3.3b} and obviously we also have
\beq \lb{4.3d}
Y_w(t)\ge Z_w(t)-\til C_2
\eeq
for some $f,w$-independent $\til C_2$.
We next notice that \eqref{4.2} gives $w_0(x)\ge v(x)$, thus $w(t,x)\ge u(t,x)$ and so 
\beq \lb{4.3}
 Z^\pm_{w,\eps}(t)\ge Z^\pm_{u,\eps}(t) \qquad\text{and}\qquad  Z_{w}(t)\ge Z_{u}(t) 
\eeq 
for all $t\ge 0$. It is therefore sufficient to show that there is a $f,u,w$-independent $t_1$ such that
\beq \lb{4.3c}
Y_w(t)\le Y_u(t+t_1)
\eeq
for all $t\ge 0$, because then \eqref{3.1} and \eqref{3.4b} (with new $f,u,w$-independent constants) follow from \eqref{4.3d}, \eqref{4.3} and \eqref{3.3c}, \eqref{3.4a}, \eqref{3.4c}.

We prove this by bounding $w$ above by a uniformly bounded time shift of $u$ plus a small perturbation. The argument is very similar to that in the proof of Lemma \ref{L.3.2}(i), with $\mu$ in place of $\lambda_\zeta/2$. We let
\begin{align*}
b_0 & \equiv \frac { \eps_0 \mu(c_0-c_\zeta)   \omega}{2\Omega K \sup_D \Phi(0,x)} \qquad (\le \eps_0 \text{ due to \eqref{3.8a} and $\mu\le \lambda_\zeta/2$}), \\
Y_0 & \equiv Y+ \frac {|\log b_0|}\mu,
\end{align*}
and choose $t_0$ so that 
 \beq \lb{4.3a}
Z_u(t_0)\ge Y_0+C_{\eps_0}+C_2.
\eeq
 This can be done uniformly in $f,u,w$ thanks to \eqref{3.4c} and $Z_u(0)$ depending only on $v$. Let
\begin{align*}
\beta(t) & \equiv \frac {\eps_0} {\Omega} \left( 1-e^{-\mu(c_0-c_\zeta)(t-t_0)/2} \right), \\
\phi(t,x) & \equiv b_0 \Phi(x_1-Y_0-c_\zeta(t-t_0),x), \\
z(t,x) & \equiv \til u(t,x) +\phi(t,x) \equiv u(t+\beta(t),x)+\phi(t,x),
\end{align*}
so that \eqref{4.2b} and \eqref{4.1} give 
\beq \lb{4.4}
z(t_0,x)\ge\phi(t_0,x)\ge w_0(x)
\eeq
 for all $x\in D$. Convexity of $\kappa(\lambda)$, $\kappa(0)=0$, and $\mu<\lambda_\zeta$ yield
\[
 \kappa(\mu) \le \frac \mu{\lambda_\zeta} \kappa(\lambda_\zeta)  <  c_\zeta \mu,
\]
so that again $\phi_t+q\cdot\nabla\phi - \divg(A\nabla \phi) \ge 0$.

It then follows, as in the proof of Lemma \eqref{L.3.2}(i), that $z$ is a supersolution of \eqref{1.1}. The argument is  identical, with $\eps=\eps_0$, $\mu$ in place of $\lambda_\zeta/2$, \eqref{3.14a} replaced by 
\beq \lb{4.3b}
Z_u(t+\beta(t))-Y_0 -c_\zeta(t-t_0) \ge \frac{c_0-c_\zeta}2 (t-t_0) + C_2
\eeq
(which is immediate from \eqref{4.3a} and \eqref{3.4c}), and \eqref{3.14b} by
\beq \lb{4.3e}
\phi(t,x)\le b_0 e^{-\mu [ (c_0-c_\zeta)(t-t_0)/2 + C_2]} \sup_D \Phi(0,x) = \frac {\beta'(t)\omega} K e^{-\mu C_2} \le \eps_0 \frac{\mu(c_0-c_\zeta)}{2K} \le \eps_0
\eeq
for $x_1\ge Z_u(t+\beta(t))$.

Thus \eqref{4.4} yields $z(t+t_0,x)\ge w(t,x)$ for all $t\ge 0$ and $x\in D$.
But then  $\Psi(s,x)\le \tfrac 12 \Phi(s,x)$ for $s\ge 0$ and $b_0\le \tfrac 12$ give
\[
w(t,x)\le \til u(t+t_0,x) + \phi (t+t_0,x) \le \frac 12 \Phi \left(x_1-Y_u \left(t+t_0+\beta(t) \right),x\right) + \frac 12 \Phi \left(x_1-Y_0-c_\zeta(t-t_0),x \right)
\]
for $x_1\ge Y_u (t+t_0+\beta(t) )$. Since $Y_u (t+t_0+\beta(t) )\ge Y_0+c_\zeta(t-t_0)$  (by \eqref{4.3b} and \eqref{3.4a}) and $\Phi(s,x)\ge 1$ for $s\le 0$, we obtain $w(t,x)\le \Phi (x_1 - Y_u (t+t_0+\beta(t)),x)$ for $t\ge 0$ (and all $x_1$). This and $\beta(t)\le \eps_0/\Omega$ now yield \eqref{4.3c} with $t_1\equiv t_0+\eps_0/\Omega$.
 \end{proof}

\begin{lemma} \lb{L.4.2}
Lemmas \ref{L.3.2}, \ref{L.3.3}, \ref{L.3.4} hold for $u,w$ as above, and the convergence in Lemma \ref{L.3.4} is uniform in all $f,u,w$ as above (with fixed $q,A,f_0,f_1,\zeta,g,K,\tht''$).
\end{lemma}

\begin{proof}
This is identical to the  proofs of Lemmas \ref{L.3.2}--\ref{L.3.5}, the only change being the replacement of $\lambda_\zeta$ by $\mu$ in the proof of Lemma \ref{L.3.2}.
\end{proof}

This and $f$-uniform convergence of $u$ to a time shift of $w_+$  prove $f,w$-uniform convergence of $w$ to a time shift of $w_+$ in $L^\infty_x$, claim (a) in Theorem \ref{T.1.1}(iii). 

The proof of claim (b) is virtually identical, with a separate treatment of the two reaction zones of $w$ (one on either side of $x_1=a$) moving right and left. This requires the adjustment of the definition of  $Z^-_{w,\eps}(t)$ (for the right-moving reaction zone) to
\[
Z^-_{w,\eps}(t) \equiv \sup \{y\ge a \,\big|\, w(t,x)\ge 1-\eps \text{ when } x_1\in [a,y]\},
\]
and a restriction of all the estimates to $x_1\ge a$. The rest of the proof is unchanged because our subsolution $z_-$ in Lemma \ref{L.3.2}(ii) is in fact negative for $t\ge t_0$ and  $x_1<a$, as long as $t_0$ is large enough (depending on $\eps$) so that $\phi_-(t,x)\ge 1$ for these $t,x_1$ (see \eqref{3.15a}). Thus we still obtain $z_-\le w$ and ultimately prove $f,w$-uniform convergence in $L^\infty_x(D_a^+)$ of $w$ to a time shift of $u$ (and hence of $w_+$), with $D^+_a\equiv[a,\infty)\times\bbT^{d-1}$. A similar treatment of the left-moving reaction zone of $w$ gives a $f,w$-uniform convergence in $L^\infty_x(D_a^-)$ of $w$ to a time shift of $w_-$, with $D^-_a\equiv (-\infty,a]\times\bbT^{d-1}$. Since $w_\pm$ converge $f$-uniformly to 1 in $L^\infty_x(D_a^\mp)$ by Lemma \ref{L.2.6}, the claim follows.

\medskip
\appendix
\section*{Appendix. The Spreading Lemma and Transition Fronts for Homogeneous Ignition Reactions} \lb{SA}
\renewcommand{\theequation}{5.\arabic{equation}}
\renewcommand{\thetheorem}{5.\arabic{theorem}}
\setcounter{theorem}{0}
\setcounter{equation}{0}

We will now show how one can use our arguments to obtain a proof of Lemma \ref{L.2.6}(ii), which is from \cite{Xin}, without the use of \cite{Xin}. In fact, we will prove a slightly stronger result. In the course of its proof we will also prove Theorem \ref{T.1.1} for $f(x,u)=f_0(u)$, showing that $c_0$ in that theorem is well defined. Recall that $c_0,c_0^->0$ are the speeds of the unique right- and left-moving fronts for \eqref{5.1} below.

\begin{lemma} \lb{L.5.1}
Let $q,A,f_0$ be as in (H1),(H2).
Then for each $\nu>0$ there is $L_\nu>0$  such that for every $\eps>0$ there is $t'_\eps<\infty$ satisfying the following. If $u:[0,\infty)\times D\to [0,1]$ solves 
\beq \lb{5.1}
u_t + q(x)\cdot\nabla u = \divg(A(x)\nabla u) + f_0(u),
\eeq
 and $\til x\in D$ is such that  $\inf \big\{ u(0,x) \,\big|\, |x_1-\til x_1|\le L_\nu \big\} \ge \tht+\nu$, then for each $t\ge 0$ we have
\beq \lb{5.2}
\inf \big\{ u(t+ t'_\eps,x) \,\big|\, x_1-\til x_1\in [-c_0^- t, c_0t] \big\} \ge 1-\eps.
\eeq
\end{lemma}

{\it Remark.} The comparison principle then gives the same for solutions of \eqref{1.1} with $f\ge f_0$. This also gives Lemma \ref{L.2.6}(i) with $c_0, c_0^-$ in place of $c_0-\eps, c_0^- -\eps$. Thus $(c_0+c_\zeta)/2$ can be replaced by $c_0$ in \eqref{3.4b} and  \eqref{3.4c}, proving the first claim in Remark 1 after Theorem \ref{T.1.1}.

\begin{proof}
Let $f(x,u)\equiv f_0(u)$ and let us reprove Theorem \ref{T.1.1} without relying on Lemma \ref{L.2.6} as originally stated. We first let $\til\tht=\tht+\nu$ and construct a compactly supported initial datum $v(x)\le\til\tht$ that satisfies \eqref{2.1}. This is done as in Lemma \ref{L.2.1} but cutting off $v$ on both sides. We let $\til v_+$ be as $\til v$ in that lemma and $\til v_-$ be a $C^2$ solution of $-\divg(A\nabla \til v_-)+q\cdot \nabla\til v_- = -q_1+\divg(Ae_1)$ on $\bbT^d$, periodically continued to $D$.  Let $v_{\pm,\eps(x)} \equiv \eps(\til v_\pm(x)\mp x_1)$ and $v(x)\equiv \rho( \min\{ v_{+,\eps}(x), v_{-,\eps}(x)+4 \} )$ for $\eps>0$. So $v$ is compactly supported and if $\eps$ is small, then there is $a\in\bbR$ and $l>0$ such that 
\[
v(x)=\begin{cases} \til\tht & x_1\in[a-l,a+l], \\ \rho( v_{+,\eps}(x)) & x_1\ge a, \\ \rho( v_{-,\eps}(x)+4) & x_1\le a \end{cases}
\]
($a$ is such that $v_{\pm,\eps}(x)\approx \pm 2$ when $x_1\approx a$).
We also have $-\divg(A\nabla v_{\pm,\eps})+q\cdot \nabla v_{\pm,\eps} =0$, so for some distribution $T\ge 0$ supported on the set $D_\eps\equiv\{x\in D\,|\, v_{+,\eps}(x)=0 \text{ or } v_{-,\eps}(x)=0\}$ and with $\til v$ standing for $\til v_\pm$ when $\pm(x_1-a)\ge 0$,
\[
-\divg(A\nabla  v)+q\cdot \nabla v = -\eps^2 \chi_{D\setminus D_\eps} \rho''(\min\{ v_{+,\eps}(x), v_{-,\eps}(x)+4 \}) (\nabla \til v-e_1) \cdot A (\nabla \til v-e_1)  - T,  
\]
which is again less than or equal to $f_0(v)$ if $\eps$ is small enough.

Next let $L_\nu<\infty$ be such that $v$ is supported in $[-L_\nu,L_\nu]\times\bbT^{d-1}$. Let $u$ solve \eqref{5.1} with $u(0,x)=v(x)$, so that Lemma \ref{L.2.2} holds for $u$ and $t>0$. Thus  there is $\tau_\nu<\infty$ such that 
\[
u(\tau_\nu,x)\ge \til\tht \chi_{[-L_\nu-p,L_\nu+p]}(x_1)\ge \max \{v(x-pe_1),v(x),v(x+pe_1) \}. 
\]
This, \eqref{2.3} for $u$, and the comparison principle then prove Lemma \ref{L.5.1} with $c_0'\equiv p/\tau_\nu>0$ in place of $c_0,c_0^-$. This in turn proves Lemma \ref{L.2.6} with $c_0'$ in place of $c_0-\eps, c_0^- -\eps$. 

We now notice that $f_0$ being independent of $x$ and positive on $(\tht,1)$ shows that $f_0$ $\zeta$-majorizes some ($\zeta$-dependent) $g$ as in Theorem \ref{T.1.3} for each $\zeta>0$. We therefore choose $\zeta>0$ small enough (and a corresponding $g$) so that $c_\zeta<c_0'$ in \eqref{2.8c} and $f_0$ $\zeta'$-majorizes $g$ for all $\zeta'>\zeta-\sigma$ (with $\sigma\equiv \zeta/2$). This can be done because $\kappa(0)=\kappa'(0)=0$ \cite[Proposition~5.7(iii)]{BH}.  Now we can perform the rest of the proof of Theorem \ref{T.1.1} for \eqref{5.1} using $c_0'$ in place of $c_0,c_0^-$ (and, in particular, $c_\zeta<c_0'$ in place of $c_\zeta<c_0$). This yields the existence of a unique right-moving transition front for \eqref{5.1} with some speed $c_0>0$, and similarly a left-moving one with speed $c_0^->0$, as well as convergence of general solutions to them as in Theorem~\ref{T.1.1}(iii). So the solution $u$ above converges in $L^\infty_x$ to some time shifts of these fronts in the sense of \eqref{1.5}. This and the comparison principle now proves \eqref{5.2}.
\end{proof}



\end{document}